\documentclass[11pt]{article}
\usepackage{amsmath, amssymb, theorem, latexsym, epsfig, graphics, eufrak, floatrow}
\usepackage{subfigure}
\PassOptionsToPackage{linktocpage}{hyperref}
\usepackage[hyperindex,breaklinks]{hyperref}
\usepackage[a4paper, total={6in, 8in}]{geometry}
\numberwithin{equation}{section}
            
\theoremstyle{plain}
\theorembodyfont{\itshape}
\newtheorem{theorem}{Theorem}[section]
\newtheorem{proposition}[theorem]{Proposition}
\newtheorem{lemma}[theorem]{Lemma}
\newtheorem{corollary}[theorem]{Corollary}
                                                                                
\theorembodyfont{\rmfamily}
\newtheorem{definition}[theorem]{Definition}

\newtheorem{example}[theorem]{Example}
\newtheorem{remark}[theorem]{Remark}

\newtheorem{problem}[theorem]{Problem}
\newtheorem{conjecture}[theorem]{Conjecture}
\newenvironment{proof}{{\noindent \textbf{Proof}\,\,}}{\hspace*{\fill}$\Box$\medskip}
\def\la{\lambda}

\def\cc{\mathbb C}
\def\oc{\overline\cc}

\def\rr{\mathbb R}

\def\mct{\mathcal T}
\def\zz{\mathbb Z}
\def\La{\Lambda}
\def\cp{\mathbb{CP}}
\def\diag{\operatorname{diag}}
\def\sign{\operatorname{sign}}
\def\mcl{\mathcal L}
\def\triv{\operatorname{Triv}}

\def\re{\operatorname{Re}}

\def\mcx{\mathcal X}
\def\mcg{\mathcal G}
\def\mch{\mathcal H}

\def\nn{\mathbb N}
\def\mca{\mathcal A}
\def\mcb{\mathcal B}

\def\mcc{\mathcal C}
\def\mcp{\mathcal P}
\def\wh#1{\widehat#1}

\def\mcr{\mathcal R}
\def\mca{\mathcal A}
\def\mcb{\mathcal B}
\def\idd{\operatorname{Id}}

\def\wt#1{\widetilde#1}

\def\tr{\operatorname{Tr}}
\def\tt{\mathbb T}
\def\mcy{\mathcal Y}
\def\im{\operatorname{Im}}
\def\gl{\operatorname{GL}}

\def\psl{\operatorname{PSL}}
\def\mce{\mathcal E}

\def\mo{\operatorname{mod}}

\def\mcf{\mathcal F}
\def\mcy{\mathcal Y}

\def\gga{\mathbf\Gamma}
\title{On extended model of  Josephson junction, linear systems with polynomial solutions, determinantal surfaces and Painlev\'e III equations}
\author{Alexey Glutsyuk\thanks{CNRS, UMR 5669 (UMPA, ENS de Lyon), France. E-mail: 
aglutsyu@ens-lyon.fr} \thanks{HSE University, Moscow, Russia} \thanks{Higher School of Modern Mathematics MIPT, 1 Klimentovskiy per., Moscow, Russia}\thanks{Work supported by the Theoretical Physics and Mathematics Advancement Foundation ``BASIS''}}
\begin{document}
\maketitle
\centerline{\it To dear Victor Matveevich Buchstaber on the occasion of his 80-th birthday}
\begin{abstract} We consider a three-parameter family of linear special double confluent Heun equations introduced and studied by V.M.Buchstaber and S.I.Tertychnyi, which is an equivalent presentation of  
a model of  Josephson junction in superconductivity. 
Buchstaber and Tertychnyi have shown that the set of those complex 
parameters for which the 
 Heun equation has a polynomial solution is a union of the so-called spectral curves: 
 explicit algebraic curves in $\cc^2$ indexed by  $\ell\in\nn$. In his paper with I.V.Netay, the author have shown   that each spectral curve is irreducible in  Heun equation parameters (consists of two irreducible components in  parameters of Josephson junction model).  
Netay discovered numerically and conjectured a genus formula for spectral curves. He reduced it to the conjecture stating that each of them is regular in 
$\cc^2$ with a coordinate axis deleted. Here we prove Netay's regularity  and genus  conjectures. For the proof we study a four-parameter family of linear systems on the Riemann sphere extending a family of linear systems equivalent to the Heun equations. They yield an equivalent presentation of the extension of model of Josephson junction introduced by the 
author in his paper with Yu.P.Bibilo. We describe the so-called determinantal surfaces, which consist of  linear systems with polynomial solutions, as explicit affine algebraic hypersurfaces in $\cc^3$ indexed by  $\ell\in\nn$. The spectral curves are their intersections with the hyperplane 
corresponding to the initial model. We prove that  each determinantal surface  is regular outside appropriate hyperplane and  consists of two rational irreducible components. The proofs use  Stokes phenomena  theory, holomorphic vector bundle technique, the fact that  each determinantal surface  is foliated by isomonodromic families of linear systems (this foliation is governed by Painlev\'e III equation) and 
transversality of the latter foliation to the initial model.
\end{abstract}
\tableofcontents
\section{Introduction}
\subsection{Introduction, brief description of main results and of their proofs} 

The three-parameter family of special double confluent Heun equations 
  \begin{equation} z^2E''+((-\ell+1)z+\mu(1-z^2))E'+(\lambda+\mu(\ell-1)z)E=0 \label{heun2''}\end{equation}
  belongs to the well-known class of Heun equations, see \cite{sla}. It was studied 
  by V.\,M.\,Buchstaber and S.\,I.\,Tertychnyi in \cite{tert2, 
bt0, bt1, bt2, bt3,  bt4} and in joint papers \cite{bg, bg2} by Buchstaber and the author. Buchstaber and Tertychnyi  have shown in \cite{bt0, bt1} that its restriction to real parameters satisfying the inequality $\la+\mu^2>0$ is equivalent to 
a model of overdamped Josephson junction in superconductivity; see the corresponding material in Subsection 1.4. 
In \cite{bt0} they have described those complex parameters $(\ell,\la,\mu)$ with $\mu\neq0$ 
for which equation (\ref{heun2''}) has a polynomial solution: this holds exactly, when 
$\ell\in\nn$ and the point $(\la,\mu)$ lies on an algebraic curve 
$\Gamma_\ell\subset\cc^2_{\la,\mu}$ 
called the {\it spectral curve}. Namely, $\Gamma_\ell$ is the zero locus of the 
determinant of a remarkable 3-diagonal $\ell\times\ell$-matrix, see \cite[formula (21)]{bt0} 
 and Theorem \ref{tpol} below. It was shown by the author in 
\cite[theorem 1.3]{gn19} that each spectral curve $\Gamma_\ell$ is irreducible. 

One of the main result (joint with I.V.Netay) of the present paper is the formula for genus of the 
spectral curve, see Subsection 1.2. Netay discovered it numerically, stated it as a conjecture, see 
\cite[conjecture 1.5]{gn19} and reduced it to the 
conjecture saying that the curve 
$$\Gamma_\ell^o:=\Gamma_\ell\setminus\{ \mu=0\}\subset\cc^2$$ 
is  regular, 
see \cite[corollary 3.6]{gn19}. Here we prove Netay's smoothness and hence, genus conjectures. 

Family of Heun equations (\ref{heun2''}) is  equivalent to the family of linear systems 
\begin{equation} Y'=\left(\frac1{z^2}\diag(-\mu, 0)+\frac1z\left(\begin{matrix}\ell & -\frac{a}2\\ \frac{a}2 & 0\end{matrix}\right)+\diag(0, \mu)\right)Y,\label{mch0}\end{equation}
\begin{equation}\la=\frac{a^2}4-\mu^2,\label{laas}\end{equation}
in the following sense: for every solution $Y(z)=(Y_1(z),Y_2(z))$ of system (\ref{mch0}) the function 
$$E(z)=Y_2(z)$$ 
satisfies (\ref{heun2''}) and conversely: for each solution $E(z)$ of (\ref{heun2''}) the 
vector function 
\begin{equation}Y=(Y_1,Y_2), \ \ Y_2(z)=E(z), \  Y_1(z)=\frac{2z}a(E'(z)-\mu E(z))\label{e-w}
\end{equation}
is a solution of (\ref{mch0}). See an implicit equivalent statement in \cite[subsection 3.2, p. 3869]{bg}. 

\begin{remark} \label{rempol}
 A function $E(z)$  is  a polynomial solution of double confluent Heun equation 
(\ref{heun2''}), if and only if the corresponding vector function $Y(z)$ given by (\ref{e-w}) is a vector polynomial 
solution of system (\ref{mch0}). 
\end{remark}

For the proof of the genus and smoothness conjectures  we consider a four-parameter extension of the family of linear systems (\ref{mch0}) depending on an additional parameter $\chi$: the family 
\begin{equation} Y'=\left(\frac1{z^2}\left(\begin{matrix}-\frac s2 & -s\chi \\ 0 & 0\end{matrix}\right)+\frac1z\left(\begin{matrix}\ell-\chi a & -\frac{a}2\\ \frac{a}2 & \chi a\end{matrix}\right)+\left(\begin{matrix} 0 & 0 \\ s\chi  & \frac s2\end{matrix}\right)\right)Y.\label{mchoyn}\end{equation}
The initial family (\ref{mch0}) is the subfamily of (\ref{mchoyn}) with 
$\chi=0$, written in the rescaled parameter $\mu:=\frac s2$. 
The restriction of family (\ref{mchoyn}) to real parameters is equivalent to the 
four-dimensional extension of model of Josephson junction introduced  by the author 
in \cite[subsection 6.2]{bibgl}. See the corresponding material in Subsection 1.5.

Our next series of main results, presented in Subsection 1.3, concern the set of those complex parameter values in (\ref{mchoyn}) with $s\neq0$ for which system (\ref{mchoyn}) 
has a vector polynomial solution. We show that this is a union indexed by $\ell\in\nn$ of 
affine algebraic surfaces $S_\ell\subset\cc^3_{\chi,a,s}$, called {\it determinantal surfaces;} each $S_\ell$ consists of two irreducible components $S_{\ell,\pm}$ permuted by the involution $(\chi,a,s)\mapsto (-\chi,-a,s)$. Each $S_{\ell,\pm}$ will be presented as the zero locus  
of an explicit polynomial, the determinant of an $(\ell+1)\times(\ell+1)$-matrix. 
We show that the surface $S_{\ell,\pm}$ is rational and the  complement 
$$S_{\ell,\pm}^o:=S_{\ell,\pm}\setminus\{ s=0\}$$ 
is a regular surface biholomorphically parametrized by the complement of $\cc^*\times\oc$ to 
an analytic hypersurface (curve, which a priori may be neither irreducible, nor algebraic), and the 
surfaces $S_{\ell,\pm}^o$ are disjoint. The latter regularity and parametrization statements are proved in Subsections 2.8, 2.9 by using Stokes phenomena theory and holomorphic vector bundle technique analogous to 
the one used by A.A.Bolibruch in his famous works on the Riemann--Hilbert Problem and related topics: see \cite{Bol89}--\cite{Bol18} and references therein. 

Using the above results on the determinantal surfaces $S_\ell$ we prove Netay's smoothness and genus conjectures on the spectral curves $\Gamma_\ell$ in Subsection 2.11 as follows. 
The curve $\Gamma_\ell^o$ written in the coordinates $(a,s)$  
instead of $(\la,\mu)$,  $s=2\mu$, $\la=\frac{a^2-s^2}4$, see 
(\ref{laas}), is known to consist of two irreducible components $\Gamma^o_{\ell,\pm}$ that are permuted by the involution $(a,s)\mapsto(-a,s)$ and are bijectively projected  onto $\Gamma_\ell^o$ 
via the projection $(a,s)\mapsto (\la,\mu)$, as was shown in  \cite[section 3]{gn19} using results of 
\cite{gn19} and \cite[section 3]{bt0}. Smoothness of $\Gamma_\ell^o$ is equivalent to 
smoothness of any of the curves $\Gamma_{\ell,\pm}^o$. One has 
$$\Gamma_{\ell,\pm}^o=S_{\ell,\pm}^o\cap\{\chi=0\}.$$
Thus, to prove smoothness of the curve $\Gamma_{\ell,\pm}^o$, and hence, 
the genus conjecture, 
it suffices to prove transversality of the (smooth) surface $S_{\ell,\pm}^o$ and the hyperplane $\{\chi=0\}$. 

To prove transversality, we consider the line field on the parameter space introduced by the author in \cite[subsection 6.2]{bibgl}:
\begin{equation}\begin{cases}  \ell'_s=0\\
\chi'_s=\frac{a-2\chi(2\chi a-\ell)}{2s}\\
  a'_s=-2s\chi+\frac as(2\chi a-\ell)
 \end{cases}.\label{isomnews'}\end{equation} 
 It was shown by the author \cite[theorem 6.6]{bibgl} that each its integral curve 
 (graph of solution of (\ref{isomnews'})) corresponds to an isomonodromic family of linear 
 systems\footnote{In fact, this was proved in \cite{bibgl} for equivalent linear systems, which are obtained from (\ref{mchoyn}) by the space variable change $\wh Y= e^{-\frac s2z}Y$.} (\ref{mchoyn}), and 
  for every $\ell\in\cc$ the function 
  \begin{equation}w(s):=\frac{a(s)}{2s\chi(s)}\label{wsnew}\end{equation}
 satisfies Painlev\'e III equation\footnote{In system  (\ref{isomnews'}) and in Painlev\'e III equation (\ref{p3}) 
 presented in \cite{bibgl} the parameter $\ell$ is changed to $-\ell$.}   
  \begin{equation} w''=\frac{(w')^2}w-\frac{w'}{s}+2\ell\frac{w^2}{s}-(2\ell+2)\frac1s+w^3-\frac1w.\label{p3}\end{equation}
   along solutions of (\ref{isomnews'}). As was shown in loc. cit. the above-mentioned  
isomonodromic families are induced from Jimbo isomonodromic deformations given in \cite{J}.  
  We show that isomonodromic deformation preserves 
   the property of linear system (\ref{mchoyn}) to have a polynomial solution, and hence, 
   each determinantal surface $S_{\ell,\pm}^o$ is tangent to the line field on $\cc^3_{\chi,a,s}$ induced by 
    (\ref{isomnews'}).  The hyperplane $\{\chi=0\}$ is clearly transversal to line field (\ref{isomnews'}) at its points where $s,a\neq0$. 
   We show that the surfaces $S_{\ell,\pm}^o$ do not intersect the $s$-axis $\{\chi=a=0\}$. 
   This together with their tangency to field (\ref{isomnews'}) implies  transversality, and hence, smoothness of the intersections $\Gamma_{\ell,\pm}^o=S_{\ell,\pm}^o\cap\{\chi=0\}$. This will finish the proof of smoothness and genus conjectures. 
   
   The foliation of each surface $S_{\ell,\pm}^o$ by integral curves of the line field  induced by 
   (\ref{isomnews'}) is called its {\it isomonodromic foliation.} The isomonodromic foliations 
   of the surfaces $S_{\ell,\pm}^o$ with $\ell=1,2$ (and also an analogous foliation for $\ell=0$) are discussed 
   in Section 3. For $\ell=0,1$ we show that they are equivalently described by explicit Riccatti equations. 
   
   Plan of proof of main results is presented in Subsection 1.6. 
   
   Some open problems are presented in Subsection 1.7.

\subsection{Genus and smoothness of the spectral curves}

Let us recall the description of the parameters corresponding to equations (\ref{heun2''})
with polynomial solutions.
To do this, consider the three-diagonal $\ell\times\ell$-matrix 
\begin{equation}
    H_\ell=\left(
    \begin{matrix}  0 & \mu & 0 & 0  & 0 & 0 \dots & 0\\
        \mu(\ell-1) & 1-\ell & 2\mu & 0 & 0 & \dots & 0\\
        0 & \mu(\ell-2) & -2(\ell-2) & 3\mu & 0 & \dots & 0\\
        \dots & \dots & \dots & \dots & \dots & \dots & \dots\\
        0 &\dots & 0 & 0 & 2\mu & -2(\ell-2) & (\ell-1)\mu\\
        0 & \dots & 0 & 0 & 0 & \mu & 1-\ell 
    \end{matrix}\right):
    \label{defh}
\end{equation}
\begin{align*}
    H_{\ell;ij}=0 \text{ if } |i-j|\geqslant2; & \quad
    H_{\ell;jj}=(1-j)(\ell-j+1); \\
    H_{\ell;j,j+1}=\mu j; & \quad
    H_{\ell;j,j-1}=\mu(\ell-j+1). \\
\end{align*}
The matrix $H_\ell$ belongs to the class of the so called Jacobi matrices that arise in 
different questions of mathematics and mathematical physics \cite{ilyin}.

\begin{theorem}   \label{tpol}\cite[section 3]{bt0}
    A special double confluent Heun equation (\ref{heun2''})  with 
    $\mu\neq0$  has a polynomial solution, if  and only if $\ell\in\nn$ and the 
    three-diagonal matrix $H_\ell+\la\operatorname{Id}$ has zero determinant. 
    For every $\ell\in\nn$ the determinant $\det(H_\ell+\la\operatorname{Id})$ 
    is a polynomial in $(u,v)=(\la,\mu^2)$ of degree $\ell$ (which also has degree $\ell$ in 
    $(\la,\mu)$): 
    \begin{equation}
        Q_\ell(\la,\mu^2):=\det(H_\ell+\la\operatorname{Id}).
    \label{plmu}\end{equation}
\end{theorem}
See also \cite[remark 4.13]{bg2} for non-existence of polynomial solutions for 
$\ell\notin\nn$ and $s\neq0$. 
\begin{definition} For every $\ell\in\nn$ the {\it $\ell$-th spectral curve} is 
\begin{equation}\Gamma_\ell:=\{ Q_\ell(\la,\mu^2)=0\}\subset\cc^2_{\la,\mu}.\label{spcurve}
\end{equation}
\end{definition}

The following results were proved by the author in \cite{gn19}.
\begin{theorem} \label{irr} \cite[theorems 1.2 and 1.3]{gn19} For every $\ell\in\nn$ 
     the polynomial $Q_\ell(u,v)$ from (\ref{plmu}) is irreducible, and the spectral curve 
     $\Gamma_\ell$ is irreducible. 
\end{theorem}

Recall that the {\it geometric genus} of an irreducible algebraic curve is the 
genus of the Riemann surface parametrizing it bijectively (except for possible 
self-intersections), i.\,e., the genus of its normalization.

One of the main results of the paper is the following theorem, which is a joint result of the author 
with I.V.Netay.

\begin{theorem} \label{congenu}
    {\it The geometric genus of the curve $\Gamma_\ell$ equals}
    \[
        \begin{cases}
            \left(\dfrac{\ell-2}{2}\right)^2, & \ell\text{ even;} \\
            \\
            \dfrac{\ell-1}{2}\dfrac{\ell-3}2, & \ell\text{ odd.} \\
        \end{cases}
    \]
\end{theorem}

I.V.Netay has discovered the above genus formula numerically and stated it as a conjecture, see 
\cite[conjecture 1.5]{gn19}. He reduced it to the conjecture stating that the curve 
$$\Gamma_\ell^o:=\Gamma_\ell\setminus\{ \mu=0\}\subset\cc^2$$ 
is  regular, 
see \cite[corollary 3.6]{gn19}. Here we prove this regularity conjecture, which together with 
\cite[corollary 3.6]{gn19} implies the genus formula:
\begin{theorem} \label{threg} The curve $\Gamma_\ell^o$ is regular for every $\ell\in\nn$.
\end{theorem}
\begin{remark} Equation (\ref{plmu})  
defining the curve $\Gamma_\ell$ belongs to a remarkable 
 class of determinantal representations of plane curves: equations 
 $$\det(x_1L_1+x_2L_2+x_3L_3)=0,$$ 
 where $(x_1:x_2:x_3)\in\cp^2$ and $L_1$, $L_2$, $L_3$ are $\ell\times\ell$-matrices. Determinantal representations of curves  
 arise in algebraic geometry and integrable systems,  see 
 \cite{vin0, vin} and references therein. Complete description of 
 determinantal representations of smooth complex irreducible projective 
  curves was obtained in \cite{vin0}. Self-adjoint determinantal 
  representations of real smooth plane curves were described in 
  \cite{vin}.
\end{remark}

\subsection{The determinantal surfaces:  formula,  smoothness and parametrization}

\begin{definition} The {\it degree} of a polynomial vector function is the maximal degree of its component.
\end{definition}
\begin{proposition} \label{poldeg} Let $s\neq0$. If a vector polynomial solution of system (\ref{mchoyn}) exists, 
then $\ell\in\zz_{\geq0}$,  the polynomial solution has  degree $\ell$ and is 
unique up to constant factor.
\end{proposition}
Proposition \ref{poldeg} will be proved in Subsection 2.2. 
\begin{proposition} \label{pel0} Let $\ell=0$ and $(a,s)\neq(0,0)$. A system (\ref{mchoyn}) has a polynomial solution, if 
and only if  $\chi=\pm\frac12$: then the solution is a constant vector $(1,\mp1)$ up to constant 
factor.
\end{proposition}
\begin{proof}  If a polynomial solution exists, 
then it is constant (Proposition \ref{poldeg}). For $(a,s)\neq(0,0)$, $\ell=0$, the intersection of kernels of the matrices in (\ref{mchoyn}) is the subspace of 
those vectors $(v_1,v_2)$ that satisfy the system of linear equations $\frac12v_1+\chi v_2=0$, 
$\chi v_1+\frac12v_2=0$. The latter system has a nonzero solution, if and only if its determinant 
$\frac14-\chi^2$ vanishes, i.e., $\chi=\pm\frac12$, and then the solution is 
$(1,\mp1)$ up to constant factor.
\end{proof}
\begin{remark} Let $s=0$. Then system (\ref{mchoyn}) is Fuchsian and has two singular points: zero and 
infinity. It has a polynomial solution, if and only if its residue matrix has at least one integer non-negative eigenvalue. In the case, when the residue matrix is diagonalizable and both its eigenvalues are integer non-negative,  the system has two linearly independent polynomial solutions. 
\end{remark}

For every $\ell\in\nn$ let us introduce the following $(\ell+1)\times(\ell+1)$-matrix functions:
\begin{equation} G_{1,\ell}=\left(\begin{matrix}\frac s2 & \chi a-1 & 0  & \dots & 0 & 0\\
0 & \frac s2 & \chi a-2 & \dots & 0 & 0\\
\dots & \dots & \dots & \dots &\dots & \dots\\
0 & 0 & 0 & \dots & \frac s2 & \chi a-\ell\\
0 & 0 & 0 & \dots & 0 & \frac 12\end{matrix}\right),\label{mg1}\end{equation} 
\begin{equation}G_{2,\ell}=\left(\begin{matrix} 0 & 0 &\dots 
& 0 & \frac a2 & \chi s\\
0 & 0 &\dots & \frac a2 & \chi s & 0\\
\dots & \dots & \dots &\dots &\dots &\dots\\
\frac a2 & \chi s & \dots & 0 & 0 & 0\\
\chi  & 0 & \dots & 0 & 0 & 0\end{matrix}\right),\label{mg2}\end{equation}
The next three  theorems and corollary are  main results of the paper.
\begin{theorem} \label{tmatr}  System (\ref{mchoyn}) with $\ell\in\nn$ and $s\neq0$ has a vector polynomial solution, if and only if 
$$\mcp_{\ell,\pm}(\chi,a,s):=\det(G_{1,\ell}\pm G_{2,\ell})=0 \text{ for some choice of sign } \pm.$$
\end{theorem}

\begin{theorem} \label{tirred} 1) The above determinants $\mcp_{\ell,\pm}$ are irreducible polynomials in $(\chi, a, s)$ that are obtained one from the other by sign change of two variables: 
$(\chi,a)\mapsto(-\chi,-a)$. 

2) Their zero loci 
\begin{equation} S_{\ell,\pm}:=\{\mcp_{\ell,\pm}(\chi,a,s)=0\}\subset\cc^3_{\chi,a,s},\label{slpm}\end{equation}
which are called the {\bf determinantal surfaces,} are irreducible affine surfaces permuted by 
the involution 
$$\mathcal I:(\chi,a,s)\mapsto(-\chi, -a, s).$$
\end{theorem}
Set 
$$S_{\ell,\pm}^o:= S_{\ell,\pm}\setminus\{ s=0\}.$$
\begin{theorem} \label{tsmooth} 
1) The surfaces $S_{\ell,\pm}^o$ are smooth and disjoint, and they are disjoint from the $s$-axis, 
i.e., the line $\{\chi=a=0\}$.

2) Each surface $S_{\ell,\pm}^o$ is fibered over the $s$-axis (punctured at the origin) by smooth rational curves of degree $2\ell+1$ conformally parametrized by  Riemann spheres punctured in at most $2\ell+1$ points;  their projective closures in $\cp^3\supset\cc^3_{\chi,a,s}$ are  parametrized by $\oc$.

3) Each surface $S_{\ell,\pm}^o$ is biholomorphically parametrized by the complement of the product 
$\cc^*_s\times\oc$ to an analytic hypersurface (curve) so that the parametrization preserves the coordinate 
$s$: sends each $s$-fiber $\{ s\}\times\oc$  to the $s$-fiber of the surface $S_{\ell,\pm}^o$.
\end{theorem}

\begin{corollary} The projective closures in $\cp^3$ of the 
determinantal surfaces $S_{\ell,\pm}\subset\cc^3\subset\cp^3$ are rational surfaces.
\end{corollary}

\begin{proof} The projection $\pi_s:S_{\ell,\pm}\to\cc_s$ is rational, and the preimage of 
every $s\neq0$ is a  rational curve, by Statement 2) of Theorem \ref{tsmooth}. 
This together with \cite[theorem III.4]{beauville} implies that $S_{\ell,\pm}$ is a rational surface.
\end{proof}

 Let us now check that the intersection of each surface $S_{\ell,\pm}\subset\cc^3_{\chi,a,s}$ with the hyperplane 
$\{\chi=0\}$ is a component $\Gamma_{\ell,\pm}$ of the pullback of the spectral curve $\Gamma_\ell$ under the $2:1$ projection $(a,s)\mapsto(\la=\frac{a^2-s^2}4, \mu=\frac s2)$.  On one hand, 
this follows "a posteriori" from Theorem \ref{tmatr} and Remark \ref{rempol}. On the other hand, 
we will present its next direct and self-contained proof. 
To do this, let us rescale the coordinates $(a,s)$ to the new coordinates $(\mu,r)$:
$$\mu=\frac s2, \ r=\frac a2; \  \ \la=r^2-\mu^2.$$
Consider the following $\ell\times\ell$-matrices:
$$\mathcal{G}_\ell =
	\begin{pmatrix}
		0 & \ldots & 0 & \mu \\
		\vdots && \rotatebox{75}{$\ddots$} & -(\ell-1) \\
		0 &\rotatebox{75}{$\ddots$} & \rotatebox{75}{$\ddots$} & \\
		\mu & -1& & 0\\
	\end{pmatrix}.$$
The following relation (found in~\cite[section 3,~eq.~(30)]{bt0}) holds:
$$ H_\ell + (r^2 - \mu^2)\operatorname{Id} =
    -(\mathcal{G}_\ell + r\operatorname{Id})
     (\mathcal{G}_\ell - r\operatorname{Id}),$$
     and it yields that 
\begin{equation}Q_\ell(\la,\mu^2)=(-1)^\ell Q_{\ell,+}(\mu,r)Q_{\ell,-}(\mu,r),\ \ 
Q_{\ell,\pm}(\mu,r):=\det(\mathcal{G}_\ell \pm r\operatorname{Id}).\label{qpm}\end{equation}
This implies that the preimage of the curve $\Gamma_\ell\subset\cc^2_{\la,\mu}$ under the above 
projection consists of two curves  
$$\Gamma_{\ell,\pm}=\{ Q_{\ell,\pm}(\mu,r)=0\}.$$
They are irreducible, as is $\Gamma_\ell$, since the projection is $2:1$, see \cite[section 3]{gn19}. 

\begin{proposition} \label{det=} The restriction to the hyperplane $\{\chi=0\}$ of the polynomial $\mcp_{\ell,\pm}$ 
coincides with the polynomial $Q_{\ell,\pm}$ up to non-zero constant factor. In particular, one has $S_{\ell,\pm}\cap\{\chi=0\}=\Gamma_{\ell,\pm}$. 
\end{proposition}
\begin{proof} Substituting $\chi=0$ to the matrices $G_{1,\ell}$ and $G_{2,\ell}$ crosses out the 
$\chi a$ in $G_{1,\ell}$ and transforms the matrix $G_{2,\ell}$ to the matrix whose lower line 
is zero and 
whose only non-zero elements are those equal to $\frac a2$. Therefore, each determinant polynomial 
$\mcp_{\ell,\pm}$ becomes $\frac12$ times the determinant  of the $\ell\times\ell$-submatrix in 
$G_{1,\ell}\pm G_{2,\ell}$ formed by the first $\ell$ lines and columns. The latter submatrix is 
obtained from the matrix $\mcg_\ell\pm r\idd$ by reflection transposition of lines, permuting the first 
and the last lines, the 2nd and the $\ell-1$-th line etc. This permutation multiplies the determinant by 
a power of $-1$. Finally, the restriction to the hyperplane $\{\chi=0\}$ of 
the determinants  $\mcp_{\ell,\pm}=\det(G_{1,\ell}\pm G_{2,\pm})$ and $Q_{\ell,\pm}=
\det(\mathcal{G}_\ell \pm r\operatorname{Id})$ differ by a non-zero constant factor. 
This proves the proposition.
\end{proof}

\begin{example} For $\ell=1$, up to constant factor, 
 \begin{equation}\pm4\mcp_{1,\pm}(\chi,a,s):=(a\pm s)(1-4\chi^2)+4\chi;\label{mcp1}
 \end{equation}
$$S_{1,\pm}=\{\mcp_{1,\pm}=0\}=\{(a\pm s)(1-4\chi^2)+4\chi=0\}.$$
The surfaces $S_{1,\pm}$ are rational,  parametrized by the coordinates $(\chi,s)$ and permuted by the involution $\mathcal I$. This parametrization also implies rationality of the $s$-fibers of the 
surfaces $S_{1,\pm}$ given by Statement 2) of Theorem \ref{tsmooth}. 
\end{example}

\begin{example} \label{irr2} For $\ell=2$ one has 
 \begin{equation}\pm8\mcp_{2,\pm}(\chi,a,s)=(2\chi\pm1)^2(2\chi\mp1)(a^2-s^2) -2a(2\chi\pm1)(6\chi\mp1)+16\chi.\label{mcp2}\end{equation}
Each surface $S_{2,\pm}=\{\mcp_{2,\pm}=0\}$ 
is irreducible and rational. Indeed, for every fixed $\chi\in\cc$ the equation 
 $\mcp_{2,\pm}=0$ defines a conic 
 $\mcc_{\chi,\pm}\subset\cc^2_{a,s}$, whose projective closure 
 $\overline\mcc_{\chi,\pm}$ in $\cp^2_{a:s:t}\supset\cc^2_{a,s}=\{ t=1\}$ 
 passes through the point $D$ at infinity with homogeneous coordinates $[1:-1:0]$. 
The projection of the projective conic $\overline\mcc_{\chi,\pm}$ from the point $D$ to the $a$-axis yields its rational parametrization by $\oc=\cc_a\cup\{\infty\}$ that depends rationally on  $\chi$. This yields a birational parametrization of  $S_{2,\pm}$ by two parameters: 
 the parameter $\chi$ and the above projection.
   \end{example}
   \begin{remark}
  The above argument  does not prove rationality of the $s$-fibers of the surface $S_{2,\pm}$. Each its $s$-fiber $$\gamma_{s,\pm}=\{(\chi,a) \ | \ \mcp_{2,\pm}(\chi,a,s)=0\}\subset\cc^2$$ 
  is a quintic that is the zero locus of the polynomial 
 (\ref{mcp2}) with the fixed $s$. Though we prove rationality of the $s$-fibers in full generality for all 
 $\ell$, below we present yet another proof in the special case $\ell=2$, together with the description of their branching points over the $\chi$-axis. 
 \end{remark}
 
 In what follows for every $s\neq0$ by 
 $\wt\gamma_{s,\pm}$ we denote the normalization: 
 the compact Riemann surface holomorphically parametrizing the 
 projective closure $\overline\gamma_{s,\pm}$ of the $s$-fiber $\gamma_{s,\pm}$ (bijectively up to self-intersections). 
\begin{proposition} \label{pcrit} Let $\ell=2$. The projection to the coordinate $\chi$ induces 
a double branched covering $\pi_\chi:\wt\gamma_{s,\pm}\to\oc_\chi$ with two branching points (i.e., 
critical values) 
\begin{equation}\chi_{\pm,+}=\pm\frac12(1+is^{-1}), \ \ \chi_{\pm,-}= \pm\frac12(1-is^{-1}).\label{critval}
\end{equation}
In particular, the  normalization $\wt\gamma_{s,\pm}$ is the Riemann sphere, and $\gamma_{s,\pm}$ is rational. 
\end{proposition}
\begin{proof} The fact that  $\pi_\chi$ is a double covering follows from 
quadraticity in $a$ of the polynomial (\ref{mcp2}). Its branching points lie in the zero locus of its discriminant $\Delta=\Delta_s(\chi)$ of the polynomial (\ref{mcp2}) as a quadratic 
polynomial in $a$. One has 
$$\Delta_s(\chi)=4(2\chi\pm1)^2((6\chi\mp1)^2-16\chi(2\chi\mp1)+s^2(4\chi^2-1)^2)$$
$$=4(2\chi\pm1)^2(36\chi^2\mp12\chi+1-32\chi^2\pm16\chi+s^2(4\chi^2-1)^2)$$
$$=4(2\chi\pm 1)^4(1+s^2(2\chi\mp1)^2).$$
Therefore, the zeros of the discriminant are $\mp\frac12$ (of multiplicity 4) and the points (\ref{critval}) (of multiplicity one). Its zero is a branching point, if and only if its multiplicity is odd. The infinity is 
not a branching point, since the polynomial $\Delta_s$ has even degree. 
Thus, the only branching points are the points (\ref{critval}). Hence, 
$\wt\gamma_{s,\pm}$ is the Riemann sphere, by Riemann-Hurwitz Theorem. 
 Proposition \ref{pcrit} is proved.
\end{proof}

\subsection{Relation to model of  Josephson junction. Historical remarks}

In 1962 B.Josephson \cite{josephson} predicted tunnelling effect for the  {\it Josephson junction:} a system of two superconductors separated by a narrow dielectric (Nobel Prize, 1973). It was confirmed experimentally by 
P.W.Anderson and J.M.Rowell in 1963 \cite{ar}.  The   {\it overdamped Josephson junction},  
see \cite{stewart, mcc,  lev,  schmidt}, \cite[p. 306]{bar}, \cite[pp. 337--340]{lich}, 
\cite[p.193]{lich-rus}, \cite[p. 88]{likh-ulr} is described by the family of nonlinear differential equations
 \begin{equation}\frac{d\phi}{dt}=-\sin \phi + B + A \cos\omega t, \ \omega>0, \ B\geq0.\label{jos}\end{equation}
 
  Here $\phi$ is the  difference of phases (arguments) of the complex-valued 
 wave functions describing the quantum mechanic 
 states of the two superconductors. Its derivative is 
 equal to the voltage up to known constant factor. 

The variable 
and parameter changes 
\begin{equation}\tau:=\omega t, \ \theta:=\phi+\frac{\pi}2, \ \ell:=\frac B\omega, \ a=\frac1\omega, \ s:=\frac A{\omega},\label{elmu}\end{equation}
 transform (\ref{jos}) to a 
non-autonomous ordinary differential equation on the two-torus $\mathbb T^2=S^1\times S^1$ with coordinates 
$(\theta,\tau)\in\rr^2\slash2\pi\zz^2$: 
\begin{equation} \frac{d\theta}{d\tau}=a\cos\theta + \ell + s \cos \tau.\label{jostor}\end{equation}
The rotation number of dynamical system (\ref{jostor}), see \cite[p. 104]{arn}, is a function of the parameters 
$a$, $\ell$, $s$. 
For every fixed $a$, the rotation number function is considered as a function of the 
parameters $(B,A)$.  The two-dimensional phase-lock areas of family (\ref{jostor}) are those level subsets of the rotation number function in $\rr^2_{B,A}$ that have non-empty interiors. V.M.Buchstaber, O.V.Karpov and S.I.Tertychnyi have shown that they exists only for integer values 
of the rotation number \cite{buch2}. The phase-lock areas of family (\ref{jostor}) 
were studied by V.M.Buchstaber, O.V.Karpov, S.I.Tertychnyi, Yu.S.Ilyashenko, D.A.Filimonov, D.A.Ryzhov, A.V.Klimenko, O.L.Romaskevich, V.A.Kleptsyn, I.V.Schurov, Yu.P.Bibilo and the author, see \cite{bg}--\cite{bt1}, \cite{bibgl}, 
\cite{IRF, krs, RK}, \cite{g18}--\cite{gn19}, \cite{tert, tert2} and references therein. 

Equations of type (\ref{jos}), 
(\ref{jostor}) arise in different domains of mathematics, e.g., in planimeters, see \cite{Foote, foott}. A subfamily of family (\ref{jostor}) was studied by Yu.S.Ilyashenko and J.Guckenheimer \cite{ilguk} from the slow-fast system point of view. They obtained results on its limit cycles, as $\omega\to0$. 

Model (\ref{jostor}) has the following equivalent description by a family of two-dimensional linear systems of differential equations on the Riemann sphere, 
see \cite{bkt1, buch2, bt1, Foote, foott,  IRF}, \cite[subsection 3.2]{bg}. 
The variable change 
$$z=e^{i\tau}=e^{i\omega t}, \ \ \Phi=e^{i\theta}=ie^{i\phi}$$
transforms  equation (\ref{jostor}) on the function $\theta(\tau)$  
to the Riccati equation 
\begin{equation}\frac{d\Phi}{dz}=z^{-2}((\ell z+\frac s2(z^2+1))\Phi+\frac a2z(\Phi^2+1)).\label{ric}\end{equation}
Equation (\ref{ric}) is the projectivization of the two-dimensional linear system 
\begin{equation} Y'=\left(\frac{\diag(-\frac s2,0)}{z^2}+\frac{\mcb}z+\diag(0, \frac s2)\right)Y, \ \ \ 
\mcb=\left(\begin{matrix} -\ell &-\frac a2\\ \frac a2 & 0\end{matrix}\right), 
\label{tty}\end{equation}
in the following sense: a function $\Phi(z)$ is a solution of (\ref{ric}), if and only if 
$\Phi(z)=\frac{Y_2}{Y_1}(z)$, where the vector function $Y(z)=(Y_1(z),Y_2(z))$ is a solution of 
system (\ref{tty}). For $s\neq0$ system (\ref{tty}) has two irregular 
nonresonant singular points at $0$ and at $\infty$. In the new parametes 
\begin{equation}\mu:=\frac s2=\frac A{2\omega}, \ \ \la:=\frac{a^2-s^2}4=\frac1{4\omega^2}-\mu^2
\label{parchange}\end{equation}
family of systems (\ref{tty}) is equivalent to the following family of {\it special double confluent Heun equations} on the function $E=Y_2(z)$, 
see \cite{tert2}, \cite{bt0}--\cite{bt4}: 
  \begin{equation} z^2E''+((\ell+1)z+\mu(1-z^2))E'+(\lambda-\mu(\ell+1)z)E=0.\label{heun}\end{equation}
We  will deal with the so-called {\it conjugate} special double confluent Heun equation obtained from 
(\ref{heun}) by  sign change at $\ell$: 
  \begin{equation} z^2E''+((-\ell+1)z+\mu(1-z^2))E'+(\lambda+\mu(\ell-1)z)E=0.\label{heun2}\end{equation}
  Let $\rho=\rho(B,A)=\rho(B,A;\omega)$ denote the rotation number function of family (\ref{jostor}). 
 For every $r\in\zz$ let $L_r$ denote the phase-lock area of family with the rotation number $r$:
 $$L_r=\{(B,A) \ | \ \rho(B,A)=r\}\subset\rr^2_{B,A}.$$
  
  It is known that each phase-lock area  $L_r$ is an infinite garland of domains going 
  to infinity "asymptotically vertically" where every two neighbor domains are separated by one 
  point. In more detail, the boundary $\partial L_r$ consists of two graphs 
  $\{ B=g_{r,\pm}(A)\}$ of functions $g_{r,\pm}$ analytic on $\rr_A$ and having Bessel type asymptotics, 
  as $A\to\infty$ (observed and proved on physics level in ~\cite{shap}, see also \cite[p. 338]{lich},
 \cite[section 11.1]{bar}, ~\cite{buch2006}; proved mathematically in \cite{RK}). The  graphs have infinitely many intersection points (which separate adjacent domains). One of them lies in the $B$-axis, is called the {\it growth point} (it exists only for $r\in\zz\setminus\{0\}$), and its abscissa is equal to 
  $\sign r\sqrt{r^2\omega^2+1}$, see \cite[corollary 3]{buch1}. The other intersection (separation) 
  points are called {\it constrictions}. It was shown in \cite[theorem 1.4]{bibgl} that all the constrictions in  $L_r$  lie in the same vertical line 
  $$\Lambda_r:=\{ B=r\omega\},$$
   which is called the {\it axis} of the phase-lock area $L_r$. See Fig. 1.        
   \begin{figure}[ht]
  \begin{center}
   \epsfig{file=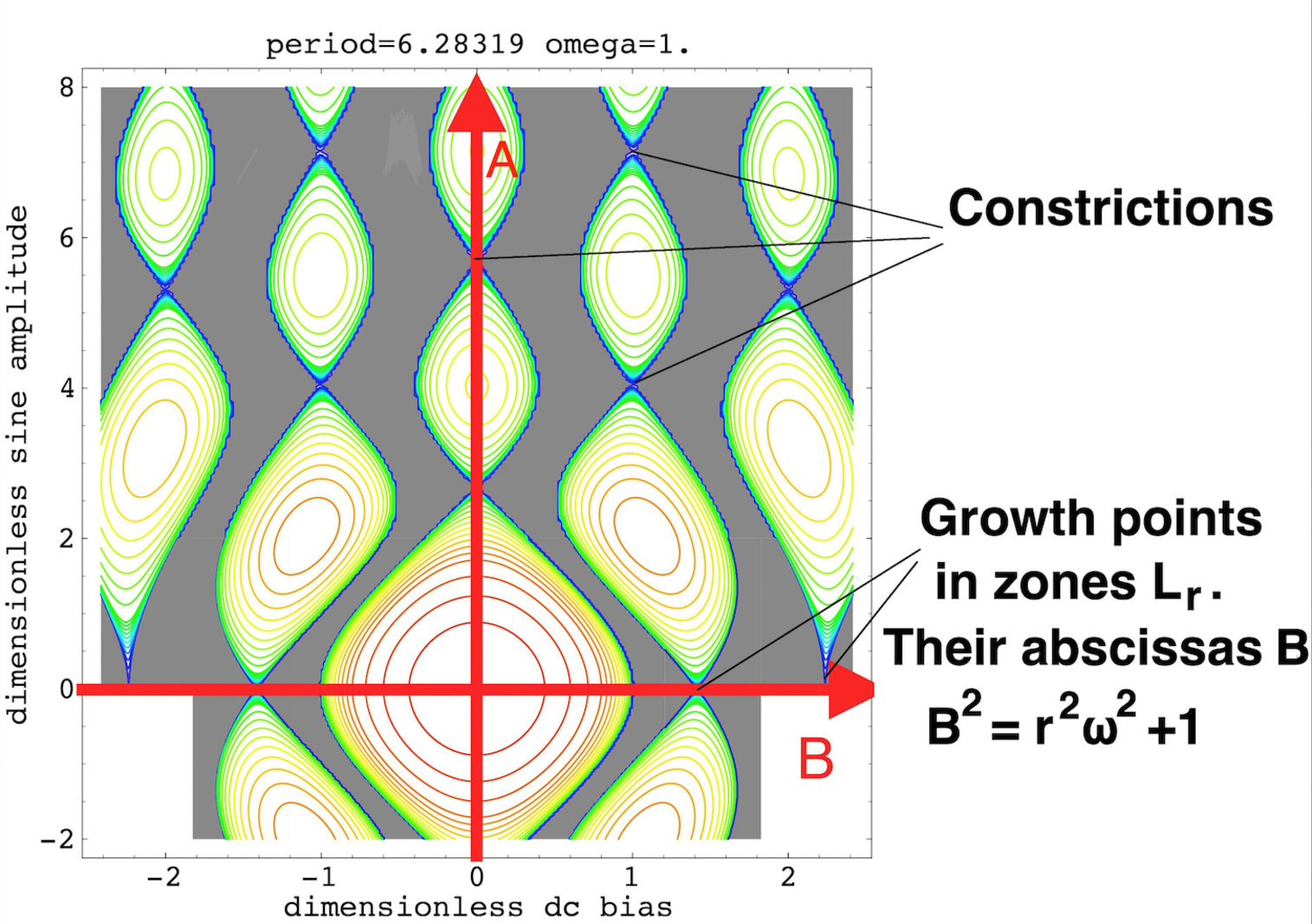, width=20em}
    \caption{Phase-lock areas, their constrictions and growth points for $\omega=1$. The abscissa is $B$, the ordinate is $A$. Figure 
    taken from paper \cite[fig. 1b)]{bg2} with authors' permission, with coordinate axes added.}
  \end{center}
\end{figure} 
 The constrictions correspond exactly to 
   those parameters $\ell=r$, $(\la,\mu)\in\rr^2$, $\la+\mu^2>0$, see (\ref{parchange}), for which the special double confluent 
   Heun equation (\ref{heun})  has an entire solution, i.e., holomorphic on $\cc_z$, as was 
   shown in \cite[theorem 3.3]{bg}.  The set of those parameters $(\la,\mu)$ for which 
   an entire solution exists is given by an explicit analytic equation in terms of an infinite matrix product  introduced in \cite{tert-prod} and in \cite[formula (30)]{bt1}. This was stated as a conjecture in loc. cit. 
    and proved in \cite[theorem 3.5]{bg}. 

    For a given $\omega>0$ a point $(B,A)\in L_r$ is called a 
    {\it generalized simple intersection}, if it is an intersection point 
    of the boundary  $\partial L_r$ with the axis $\Lambda_\ell=\{ B=\ell\omega\}$,  $\ell\in\zz$, of some phase-lock area $L_\ell$ with $\ell\equiv r(\mo2)$, and it is not a constriction. See Fig. 2. It is known that the generalized simple intersections correspond exactly to those parameters $\ell$, $\la$, $\mu$ with $\mu\neq0$, $\la+\mu^2>0$, for which the conjugate Heun equation (\ref{heun2}) has a polynomial solution, as was shown in \cite[theorem 1.15]{bg2}. The proof of this result is based 
    on the previous joint result by Buchstaber and the author \cite[theorem 3.10]{bg} stating that if  the conjugate Heun equation (\ref{heun2}) has a polynomial solution, then Heun equation (\ref{heun})  has 
    no entire solution. This was a solution of a series of conjectures stated and partially studied by V.M.Buchstaber and S.I.Tertychnyi in \cite{bt1, bt2}. 

   \begin{figure}[ht]
  \begin{center}
   \epsfig{file=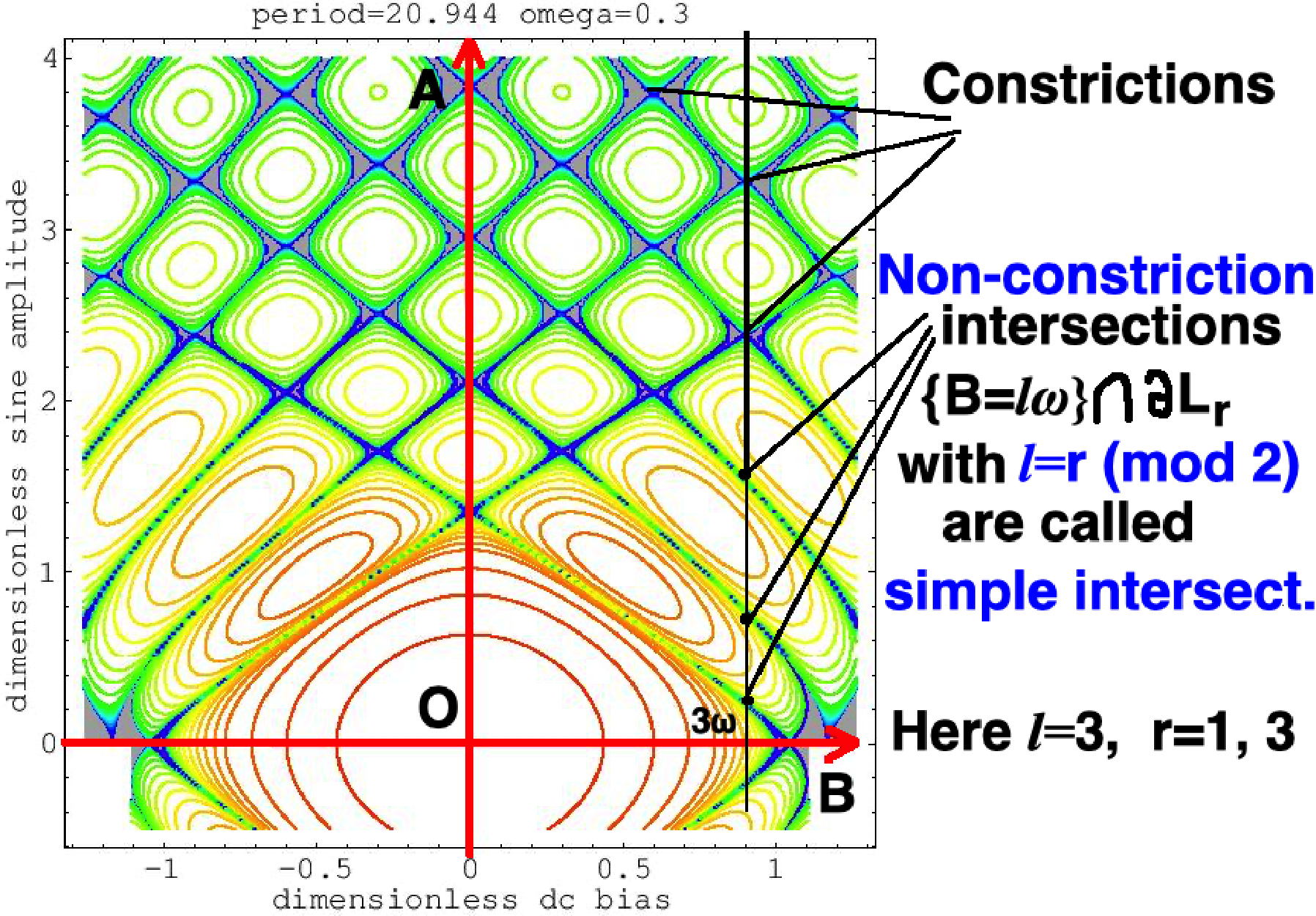, width=20em}
    \label{fig:5.1}
    \caption{Phase-lock areas for $\omega=0.3$: constrictions (correspond to Heun equations (\ref{heun}) with entire solutions) and generalized  simple intersections (correspond to conjugate Heun equations (\ref{heun2}) with polynomial solutions).}
  \end{center}
\end{figure}

\subsection{Four-parameter extension of the Josephson family of dynamical systems on 2-torus and isomonodromic foliation}
The four-parameter extension  of the three-parameter "Josephson" family of dynamical systems 
(\ref{jostor}) given in \cite[subsection 6.2]{bibgl} is 
\begin{equation}  \frac{d\theta}{d\tau}=\nu+a\cos\theta+s\cos\tau+\psi\cos(\theta-\tau); \ \ \ 
\nu,a,\psi\in\rr, \ \ s>0, \ (a,\psi)\neq(0,0).\label{gen3}\end{equation}
The variable changes 
$\Phi=e^{i\theta}$, $z=e^{i\tau}$ 
transform (\ref{gen3}) to the Riccati equation
\begin{equation}\frac{d\Phi}{dz}=\frac1{z^2}\left(\frac s2\Phi+\frac\psi2\Phi^2\right)+\frac1z\left(\nu\Phi+\frac a2(\Phi^2+1)\right) 
+\left(\frac s2\Phi+\frac\psi2\right).\label{genric}\end{equation}
A function $\Phi(z)$ is a solution of the latter Riccati equation, if and only if 
$\Phi(z)=\frac{Y_2(z)}{Y_1(z)}$, where $Y=(Y_1,Y_2)(z)$ is a solution of the linear system 
\begin{equation} Y'=\left(-s\frac{\mathbf K}{z^2}+\frac{\mathbf R}z+s\mathbf N\right)Y,\label{mchoy}\end{equation}
$$\mathbf K=\left(\begin{matrix}\frac12 & \chi \\ 0 & 0\end{matrix}\right), \ 
  \mathbf R=\left(\begin{matrix}-(\ell +\chi a) & -\frac{a}2\\ \frac{a}2 & \chi a\end{matrix}\right),  \ 
  \mathbf N=\left(\begin{matrix}0 & 0 \\ \chi  & \frac12\end{matrix}\right);
  $$
  $$\chi=\frac{\psi}{2s}, \ \ell= \nu-\frac{\psi a}{s}=\nu-2\chi a.$$
  The residue matrix of the formal normal forms of system (\ref{mchoy}) at $0$ and at $\infty$ is the same and equal to 
  \begin{equation}\diag(-\ell,0), \ \ \ell=\nu-2\chi a.\label{ellll}\end{equation}
  See the background material on formal normal forms, Stokes matrices and isomonodromic 
  deformations in Subsection 2.1. 
  \begin{theorem} \label{tisom2} \cite[subsection 6.2]{bibgl} The four-parameter family of linear systems (\ref{mchoy}) is analytically foliated by 
  one-dimensional isomonodromic families defined by the following non-autonomous system of differential equations:
  \begin{equation}\begin{cases} \chi'_s=\frac{a-2\chi(\ell+2\chi a)}{2s}\\
  a'_s=-2s\chi+\frac as(\ell+2\chi a)\\
  \ell'_s=0\end{cases}.\label{isomnews}\end{equation} 
  For every $\ell\in\rr$ the function (\ref{wsnew}), i.e., 
  \begin{equation}w(s):=\frac{a(s)}{2s\chi(s)}=\frac{a(s)}{\psi(s)}\label{wsnews}\end{equation}
  satisfies Painlev\'e III equation (\ref{p3}) with $\ell$ replaced by $-\ell$ along solutions of (\ref{isomnews}): 
  \begin{equation} w''=\frac{(w')^2}w-\frac{w'}{s}-2\ell\frac{w^2}{s}+(2\ell-2)\frac1s+w^3-\frac1w.\label{p3'}\end{equation}
  \end{theorem}
  \begin{remark} In fact, isomonodromicity was proved in loc. cit. for systems obtained 
 from (\ref{mchoy}) by adding the scalar matrix $-\frac12 Id$ to the matrix $N$ in (\ref{mchoy}). 
 But addition of the latter term changes neither projectivization (Riccati equation), nor  isomonodromicity. 
 The differential equation (\ref{isomnews}) is Hamiltonian with the 
  time $s$-depending Hamiltonian function  
  \begin{equation} H(\chi,a,s):=-\frac{\chi^2a^2}{s}+\frac{a^2}{4s}+s\chi^2-\frac{\ell\chi a}s,
  \label{hamilt}\end{equation}
  see  \cite[proposition 1.14]{g23}. 
Representations of Painlev\'e  equations as  Hamiltonian systems were found by 
J.Malmquist \cite{mlm} (for all, except Painlev\'e III) and by K.Okamoto \cite[p. 265]{ok} 
(for all, including Painlev\'e III). The above Hamiltonian representation of Painlev\'e III equation 
(\ref{p3'}) with Hamiltonian (\ref{hamilt}) can be obtained from Okamoto Hamiltonian representation 
by a variable change: see \cite[remark 1.15]{g23}. 
  \end{remark}
We will be dealing with the linear systems obtained from (\ref{mchoy}) by change of sign at $\ell$: 
  \begin{equation} Y'=\left(-s\frac{\mathbf K}{z^2}+\frac{\mathbf{\wh R}}z+s\mathbf N\right)Y,\label{mchoy*}\end{equation}
$$\mathbf K=\left(\begin{matrix}\frac12 & \chi \\ 0 & 0\end{matrix}\right), \ 
 \mathbf{\wh R}=\left(\begin{matrix}\ell -\chi a & -\frac{a}2\\ \frac{a}2 & \chi a\end{matrix}\right),  \ 
 \mathbf N=\left(\begin{matrix} 0 & 0 \\ \chi  & \frac12\end{matrix}\right).$$
 Note that for $\chi=0$ system (\ref{mchoy*}) coincides with system (\ref{tty}) where $\ell$ 
 is replaced by $-\ell$. Therefore, it is equivalent to the conjugate Heun equation (\ref{heun2}) on $E=Y_2(z)$, analogously to the similar statement on (\ref{tty}) and Heun equation (\ref{heun}), 
 see the end of Subsection 1.4. 
 
 Analogously to the above discussion, the parameter space of 
 family (\ref{mchoy*}) is foliated by isomonodromic families given by solutions of 
 system of differential equations (\ref{isomnews'}), along which the function $w(s)$ satisfies 
 Painlev\'e III equation (\ref{p3}). 

\subsection{Plan of proof of main results}

The proof of main results is based on  analytic theory of linear systems with irregular non-resonant singularities: formal normal forms, Stokes phenomena (canonical solutions in Stokes sectors and monodromy-Stokes data) 
and isomonodromic deformations. The corresponding background material is recalled in Subsection 2.1. 

In Subsection 2.2 we prove Proposition \ref{poldeg}. We will deduce it from a more general proposition 
dealing with general linear systems of type 
\begin{equation}Y'=\left(\frac K{z^2}+\frac Rz+ N\right)Y, \ \  \ K, \ R, \ N \text{ are } 2\times2-\text{matrices,}
\label{eqlin0}\end{equation}
having irregular non-resonant Poincar\'e rank 1 singular points at $0$ and at $\infty$. It states that a 
system has a vector polynomial solution, if and only if two appropriate canonical solutions coming from 
$0$ and from $\infty$ coincide and  form a solution holomorphic in $\cc$, and then it is a unique 
polynomial solution. Its degree is equal to the corresponding residue eigenvalue of the formal normal form 
at $\infty$. 

In Subsection 2.4 we prove Theorem \ref{tmatr} describing the set of parameters of those systems (\ref{mchoyn}) that have polynomial solutions as the union of the zero loci 
$S_{\ell,\pm}=\{\mcp_{\ell,\pm}=0\}$. To do this, 
in Subsection 2.3 we introduce and study a more general class of systems (\ref{eqlin0}), the so-called class S, which contains systems (\ref{mchoyn}), and show that class S coincides with the class of those systems 
(\ref{eqlin0}) that  have a symmetry of the type 
 $(Y_1(z),Y_2(z))\mapsto z^\ell(Y_2(z^{-1}),Y_1(z^{-1}))$; for systems (\ref{mchoyn}), 
 the latter $\ell$ coincides with the $\ell$ in (\ref{mchoyn}). We show that if a polynomial solution 
 of a system (\ref{mchoyn}) exists, then it is either invariant, or anti-invariant under the above symmetry, 
 and then deduce that the parameters $(\chi,a,s)$ of the system lie in $S_{\ell,+}$, respectively $S_{\ell,-}$. 
 And then we prove the converse. 
 
 For the proof of Theorem \ref{tsmooth}   we consider  analytic extensions of 
 appropriate canonical sectorial solutions  of a system (\ref{mchoyn}) at $0$ and at $\infty$ 
 to the base point $z_0=1$ along paths $\alpha_0^{-1}$ and $\alpha_\infty^{-1}$ in $\cc^*_z$ going from 
 appropriate Stokes sectors to $z_0$; the sectors and the paths  are permuted by the involution $z\mapsto z^{-1}$ (such sector and path collections will be called symmetric). Taking projectivizations of the vector values at $z_0$ of thus extended solutions yields four points 
 $q_{10}, q_{20}, q_{1\infty},q_{2\infty}\in\cp^1=\oc$. The collection of the  four points $q_{jp}$ taken together with the monodromy operator (up to natural $\operatorname{PSL_2(\cc)}$-action) is 
 called the {\it monodromy-Stokes data,} see Subsection 2.1. We consider yet another point $q_{10}'$ 
 that is the image of the point $q_{10}$ under the projectivized monodromy of system (\ref{mchoyn}) 
 acting on the space of projectivized initial conditions (local solutions) at $z_0$. In other terms, take the analytic extension of the solution $\Phi(z)$ of the Riccatti equation (the priojectivization of our system (\ref{mchoyn})) with 
 the initial condition $\Phi(z_0)=q_{10}$ along a counterclokwise circuit around the origin. Set $q_{10}'$ to be 
 the value of thus  extended solution at $z_0$.

 Our goal is to construct a parametrization 
 of each surface $S^o_{\ell,\pm}$ by  the complement of $\cc^*_s\times\oc$ to an analytic 
 hypersurface (curve).  In Subsection 2.5 we show that  (\ref{mchoyn}) has a polynomial solution, if and only if $\ell\in\nn$ and the monodromy-Stokes data is parabolic and non-degenerate. Parabolicity means that 
 some of the above points $q_{jp}$ coincide, namely, 
 $q_{20}=q_{1\infty}$, the monodromy is unipotent and its projectivization fixes $q_{20}$. Non-degeneracy means that the collection $q_{10'}$, $q_{10}$, $q_{20}$, $q_{2\infty}$ consists of at least three 
 distinct points. We describe in similar terms all the systems (\ref{eqlin0}) that are gauge equivalent to systems (\ref{mchoyn}) with polynomial solutions. 
 
 In Subsection 2.6 we show that a parabolic non-degenerate monodromy-Stokes data is uniquely determined 
 by the cross-ratio
  $$\mcr:=\frac{(q_{20}-q_{10}')(q_{10}-q_{2\infty})}{(q_{20}-q_{2\infty})(q_{10}-q_{10}')}.$$
  Then we show that two systems of type (\ref{mchoyn}) with the same $\ell$ 
   having a polynomial solution coincide up to parameter sign change $(\chi,a)\mapsto(-\chi,-a)$, if 
   and only if they have the same parameter $s$ and the same cross-ratio $\mcr$ (constructed for given 
   paths $\alpha_{0,\infty}$, the same for both systems). This already yields a local parametrization 
   of surfaces $S_{\ell,\pm}^o$ by  coordinates $(s,\mcr)\in\cc^*\times\oc$ 
   (though  we will not prove that at this place). 
   But $\mcr$ is not uniquely defined: it depends on homotopic classes of paths $\alpha_{0,\infty}$. 
   
   In order to obtain a global parametrization of surfaces $S_{\ell,\pm}^o$, in Subsection 2.7 we first construct a holomorphic family of vector bundles $\mce_{t,\mcr}=\mce_{t,\mcr,\ell}$ 
   with meromorphic connections on $\oc$ depending on two parameters 
   $(t,\mcr)\in\cc\times\oc$, set $s:=e^t$, and a continuous family of pairs of paths 
   $\alpha_{0,t}$ and $\alpha_{\infty,t}$ going from $z_0=1$ to appropriate Stokes sectors at $0$ and at 
   $\infty$ respectively and permuted by the involution $z\mapsto z^{-1}$ that satisfy the following statements:
   
  1) the connections have irregular non-resonant Poincar\'e rank 1 singular points at $0$ and at $\infty$ 
   with formal normal forms 
   \begin{equation} Y'=\left(\frac{\diag(-\frac s2, 0)}{z^2}+\frac{\diag(\ell, 0)}z\right)Y \ \ \text{ at } 0,
  \label{fnf00}\end{equation}
\begin{equation} 
 Y'=\left(\diag(0, \frac s2)+\frac{\diag(\ell, 0)}z\right)Y \ \ \text{ at } \infty;\label{fnfii}\end{equation}
 
 2) the above-mentioned Stokes sectors are those for the above normal forms; 
 
 3) the monodromy-Stokes data defined by the above paths $\alpha_{0,t}$, $\alpha_{\infty,t}$ are parabolic non-degenerate, and the corresponding cross-ratio is equal to $\mcr$;
 
 4) two bundles $\mce_{t,\mcr}$ and $\mce_{t',\mcr'}$ with connections are analytically gauge equivalent, if and only if 
 $$(t',\mcr')=(t,\mcr)+(2\pi i n,2n), \ \ n\in\zz.$$
 Parabolicity implies that each bundle with connection has a meromorphic flat section with only one pole, at infinity, which is a vector polynomial of degree $\ell$ in a trivializing chart on a neighborhood of infinity. 
 
 In Subsection 2.8 we take the quotient of the parameter space by the translation by the vector $(2\pi i,2)$, 
 which leaves invariant the coordinate $\xi:=\pi i\mcr-t$. We get a holomorphic 
 family of bundles with connections parametrized by  $(s,\xi)\in\cc^*\times\oc$, where all the bundles are distinct: pairwise gauge non-equivalent. It follows from a well-known theorem, see \cite[proposition 4.1]{malgr},  
  \cite[appendix 3]{Bol18}, \cite[theorem 2.2, p.449]{rohrl}, that the subset $\triv_\ell\subset\cc^*\times\oc$ of those parameters that correspond to holomorphically 
 trivial bundles is the complement to an analytic hypersurface (curve). They correspond to linear systems of type 
 (\ref{eqlin0}) on the Riemann sphere. Parabolicity of the monodromy-Stokes data implies that the systems in question are gauge equivalent to systems of type S having a vector polynomial solution and the above formal normal forms. 
 As is shown in Subsection 2.3, a system of type S with the above formal normal forms is gauge equivalent to a system of type (\ref{mchoyn}), 
 if and only if it satisfies the following {\it line non-coincidence condition:} 
 the eigenlines of the main term matrices at $0$ and at $\infty$ corresponding to non-zero 
 eigenvalues are distinct. Erasing yet another curve from the subset $\triv_\ell$  consisting of those systems 
 that do not satisfy the line non-coincidence condition, we get a new open and dense subset 
 $\triv^o_\ell\subset\cc^*\times\oc$ of those points that correspond to systems gauge equivalent to systems (\ref{mchoyn}) with polynomial solutions; for every point $(s,\xi)\in\triv^o_\ell$ the corresponding system (\ref{mchoyn}) is uniquely 
 defined up to sign change $(\chi,a)\mapsto(-\chi,a)$. This yields  biholomorphic parametrizations 
 of the surfaces $S_{\ell,\pm}^o$ by $\triv^o_\ell$. Afterwards we show that the complement 
 $\cc^*\times\oc\setminus\triv^o_\ell$ is an analytic curve, by proving that potential 
 non-extendability points of the closure of the second erased curve are isolated and hence, removable by Shiffman Theorem, 
 see \cite{shiffman} and \cite[subsection 4.4]{chirka}.  
 
 In Subsection 2.9 we deduce regularity of the surfaces $S_{\ell,\pm}^o$ and rationality 
 of their $s$-fibers from their parametrization and finish the proof of Theorem \ref{tsmooth}. 
 Theorem \ref{tirred} (irreducibility of surfaces $S_{\ell,\pm}$ and of the  polynomials 
 $\mcp_{\ell,\pm}$) is proved in Subsection 2.10. Smoothness of the spectral curves $\Gamma_{\ell,\pm}^o$ 
 and the genus formula are proved in Subsection 2.11.

 \subsection{Open problems}
 
 A list of open problems on model of overdamped Josephson junction, on its four-dimensional 
 extension (\ref{gen3}), its isomonodromic foliation given by (\ref{isomnews}) and related questions 
 is presented in \cite[section 6]{bibgl}. 
 
For every $r\in\zz_{\neq0}$ set 
$$L_r^+=L_r\cap\{ A>0\}, \ A_r:=\text{the ordinate of the highest simple intersection in } L_r,$$
$$\La_r:=\{ B=r\}, \ \ Sr=\La_r\cap\{ A\geq A_r\}.$$

The {\bf connectivity conjecture}, see \cite[conjecture 1.14]{g18}, states that {\it the intersection 
$L_r^+\cap\La_r$ coincides with the ray $Sr$, and thus, is connected.} By \cite[corollary 6.1 and 
remark 6.2]{bibgl}, it is reduced to a pure real-algebro-geometric conjecture on the real spectral curves 
$\Gamma_\ell$, see \cite[conjecture 6.3]{bibgl}, which we reformulate here in the following slightly stronger form.

\begin{conjecture} 1) For every $\ell\in\nn$ the intersection  of the real spectral curve $\Gamma_\ell$ defined in (\ref{spcurve}) with the domain $\{\la+\mu^2>0\}$ contains a unique connected component 
$\Gamma_\ell^*$ whose points correspond to differential equations (\ref{jostor}) on $\tt^2$ with the rotation number $\ell$. Its other components correspond to positive rotation numbers that are less than $\ell$. 

2) The projection $(\la,\mu)\mapsto\la+\mu^2$ maps the component $\Gamma_\ell^*$ 
diffeomorphically onto $\rr_+$. 
\end{conjecture}

\begin{problem} Study  real (complex) determinantal surfaces $S_{\ell,\pm}$ and their real (complex) isomonodromic foliations given by real (complex) solutions of system of differential equations (\ref{isomnews'}). 
For which real initial conditions $(a_0,s_0)$ in the hyperplane $\{\chi=0\}$ 
(corresponding to model (\ref{jostor}) of 
Josephson junction) the corresponding real solutions $(\chi(s), a(s))$ of (\ref{isomnews'}) return back 
to the hyperplane $\{\chi=0\}$ for some $s>s_0$ and for infinitely many $s>s_0$? In other words, study the definition domain of the Poincar\'e first return map (and its iterates) to the above hyperplane for the real isomonodromic foliation (\ref{isomnews'}).
\end{problem}
It is known that returns to the hyperplane $\{\chi=0\}$ corresponds to simple poles with residue 1 of the solution of Painlev\'e equation (\ref{p3}) governing the solution of (\ref{isomnews'}) in question, see \cite[remark 6.14]{bibgl}. Few solutions of Painl\'ev\'e 3 equation, 
e.g., the {\it tronqu\'ee solutions}, see \cite{lidati}, are bounded on a semi-interval $[C,+\infty)$, and 
hence, have no poles there.

\begin{problem}
 Describe those solutions of (\ref{isomnews'}) whose graphs  lie in $S_{\ell,\pm}$ and that 
 correspond to tronqu\'ee solutions 
  of the corresponding Painlev\'e III equation (\ref{p3}). 
 \end{problem}

\section{Proof of main results}

\subsection{Background material: Stokes phenomena, monodromy--Stokes data and 
 isomonodromic deformations}
 
 The following material on irregular singularities of linear systems and Stokes phenomena is contained in \cite{2, BUL, bjl, 12, jlp, sib}. 

Consider a two-dimensional linear system 
 \begin{equation}Y'=
\left(\frac K{z^2}+\frac Rz+N\right)Y, \ \ \ 
Y=(Y_1,Y_2)\in\cc^2,\label{eqlin}\end{equation}
over a neighborhood of zero in $\cc_z$. Here  $K$, $R$, $N$ are complex $2\times2$-matrices. 
We consider that $K$  has distinct eigenvalues $\la_{10}\neq\la_{20}$, and $N$ has distinct eigenvalues 
$\la_{1\infty}\neq\la_{2\infty}$. By definition, the two latter  conditions on eigenvalues are equivalent to 
the condition that the singular point 0 (respectively, $\infty$) of system (\ref{eqlin}) is 
{\it irregular non-resonant of Poincar\'e rank 1.}  The matrix $K$ is conjugate to  
$\wt K=\diag(\la_{10},\la_{20})$, $\wt K=\mathbf{H_0}^{-1}K\mathbf{H_0}$, $\mathbf{H_0}\in GL_2(\cc)$,  and one can achieve that 
$K=\wt K$ by applying the constant linear change (gauge transformation)
$Y=\mathbf{H_0}\wh Y$. 

Recall that two systems of type (\ref{eqlin}) are {\it analytically equivalent} near the origin, 
if one can be transformed to the other by linear space coordinate  change 
$Y=H_0(z)\wt Y$), where $H_0(z)$ is a holomorphic $GL_2(\cc)$-valued function on a neighborhood 
of the origin. Two systems (\ref{eqlin}) are {\it formally equivalent} at the origin, 
if the above $H_0(z)$ exists in 
the class of invertible formal power series with matrix coefficients. Analytic (formal) equivalence 
at infinity is defined analogously, with "power series" being 
Laurent series containing only non-positive degrees of $z$. 

System (\ref{eqlin}) is 
formally equivalent at the origin to a unique {\it formal normal form} 
\begin{equation}\wt Y'=\left(\frac{\wt K}{z^2}+\frac{\wt R_0}z\right)\wt Y, \ \wt K=\diag(\la_{10},\la_{20}), 
\ \wt R_0=\diag(b_{10},b_{20}),\label{nform}
\end{equation}
\begin{equation} \wt R_0 \ \  \text{ is the diagonal part of the matrix } \ \  \mathbf{H_0}^{-1}R
\mathbf{H_0}.
\label{resfo}\end{equation} 
The matrix coefficient $K$ in  system (\ref{eqlin}) and the corresponding  matrix $\wt K$ 
in (\ref{nform}) are called the 
{\it main term matrices}, and $R$, $\wt R_0$  the {\it residue matrices} at the origin. However the normalizing series $H_0(z)$ bringing (\ref{eqlin}) to (\ref{nform}) generically diverges. At the same time, there exists a covering of a punctured 
neighborhood of zero by two sectors $S_0^0$ and $S_1^0$ with vertex at 0 in which 
there exist holomorphic $GL_2(\cc)$-valued matrix functions $H_{j;0}(z)$, $j=0,1$, that are $C^{\infty}$ 
smooth on  $\overline S_j^0\cap D_r$ for some $r>0$, and such that the variable changes $Y=H_{j;0}(z)\wt Y$ 
transform (\ref{eqlin}) to (\ref{nform}). This Sectorial Normalization Theorem 
 holds for the so-called {\it Stokes sectors}. Namely, consider the rays issued from 0 and forming the set 
\begin{equation}\{\re\frac{\la_{10}-\la_{20}}z=0\}.\label{strays0}\end{equation}
 They are called {\it imaginary dividing rays} or {\it Stokes rays}.  
A sector $S_{j}^0$ is called a {\it Stokes sector,} if it contains  one imaginary dividing ray 
 and its closure does not contain the other one. 
 
 Let $W(z)=\diag(e^{-\frac{\lambda_{10}}z}z^{b_{10}},e^{-\frac{\lambda_{20}}z}z^{b_{20}})$ denote the canonical diagonal fundamental 
matrix solution of the formal normal form (\ref{nform}). The matrices 
$X^j_0(z):=H_{j;0}(z)W(z)$ are fundamental matrix solutions of the initial system (\ref{eqlin}) 
defining solution bases in $S_j^0$ called the {\it canonical sectorial solution bases.} 
Here we choose the branches $W_0(z)=W^j_0(z)$ of the  matrix function $W_0(z)$ in $S_j^0$ so that $W^1_0(z)$ is obtained from $W^0_0(z)$ by counterclockwise analytic extension from $S_0^0$ to $S_1^0$. 
And  we define the branch $W^2_0(z)$ of $W_0(z)$ in $S_2^0:=S_0^0$  obtained from $W^1_0(z)$ by counterclockwise 
analytic extension from $S_1^0$ to $S_0^0$. This yields 
another canonical matrix solution $X^2_0(z):=H_{0;0}(z)W^2_0(z)$ of system (\ref{eqlin}) in $S_0^0$, which is obtained 
from $X^0_0(z)$ by multiplication from the right by the monodromy matrix $\exp(2\pi i\wt R_0)$ 
of the formal normal form (\ref{nform}).  Let $S_{j,j+1}^0$ denote the connected component 
of the intersection $S_{j+1}^0\cap S_{j}^0$, $j=0,1$, that is crossed when one moves 
from $S_j^0$ to $S_{j+1}^0$ counterclockwise, see Fig. 3. The transition matrices $C_{00}$, $C_{10}$ 
between thus defined canonical solution bases $X^j_0$, 
\begin{equation} X^1_0(z)=X^0_0(z)C_{00} \text{ on }  S_{0,1}^0, \ \ X^2_0(z)=X^1_0(z)C_{10} \text{ on } S_{1,2}^0,\label{stokes}\end{equation}
are called the {\it Stokes matrices} at the origin. 

The above formal normal form  and sectorial normalization theorems also hold at infinity. 
Namely, the formal normal form at infinity is 
\begin{equation}\wt Y'=\left(\wt N+\frac{\wt R_\infty}z\right)\wt Y, \ \wt N=\diag(\la_{1\infty},\la_{2\infty}), 
\ \wt R_\infty=\diag(b_{1\infty},b_{2\infty}).\label{nforminf}
\end{equation}
Here $\la_{1\infty}$ and $\la_{2\infty}$ are the eigenvalues of the matrix $N$.

In more detail, let $\mathbf{H_\infty}$ be a matrix such that $\mathbf{H_{\infty}}^{-1}N\mathbf{H_{\infty}}=\diag(\la_{1\infty},\la_{2\infty})$. Then 
$$\wt R_{\infty} \  \text{ is the diagonal part of the matrix } \ \mathbf{H_\infty}^{-1}R
\mathbf{H_\infty}.$$
The Stokes rays at infinity are defined by the equation
\begin{equation}\{\re(\la_{1\infty}-\la_{2\infty})z=0\}.\label{straysinf}\end{equation}

\begin{example} Let  $K=\diag(\la_{10},\la_{20})$, and let  $\la_{20}-\la_{10}>0$. Then 
the Stokes rays are the positive and negative imaginary semiaxes. 
The Stokes sectors $S_0^0$ and $S_1^0$ covering $\cc^*$ satisfy the following 
conditions:

- the sector $S_0^0$ contains the positive imaginary semiaxis, and its closure does 
not contain the negative one;

- the sector $S_1^0$ satisfies the opposite condition. See Fig. 3.

\noindent The Stokes matrices $C_{00}$ and $C_{10}$ are unipotent upper and lower triangular 
respectively. 
Case of singular point at infinity with $\la_{2\infty}-\la_{1\infty}>0$ is treated analogously. The Stokes rays are the same, and the Stokes sectors also can be choosen the same, as above. But at infinity,  the Stokes matrices $C_{0\infty}$, $C_{1\infty}$ are oppositely-triangular: unipotent lower (respectively, upper) triangular. 
\end{example}
\begin{figure}[ht]
  \begin{center}
   \epsfig{file=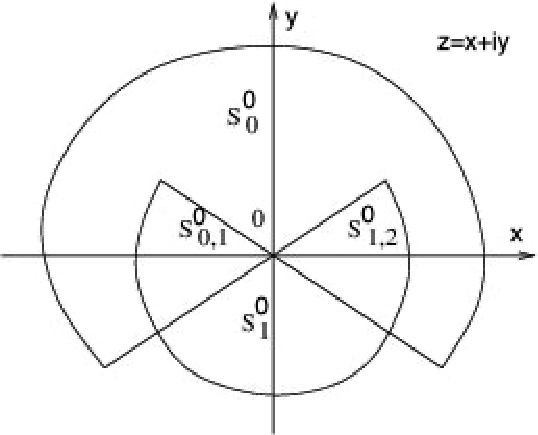}
    \caption{Stokes sectors at the origin in the case, when $\la_{10}-\la_{20}\in\rr$.}
  \end{center}
\end{figure}

Recall that the {\it monodromy operator} of system (\ref{eqlin}) acts on the space of local solutions (initial conditions) at a base point $z_0\in\cc^*$ (or on the space of solutions in a given simply connected domain 
(sector) in $\cc^*$) by analytic extension along a counterclockwise circuit around the origin. 
It is well-known that in the canonical basis in the sector $S_{0}^p$ at $p\in0,\infty$ 
the monodromy matrix is given by 
the following formula, see  \cite[p.290]{12}: 
\begin{equation} M=M_{norm,p}C_{1p}^{-1}C_{0p}^{-1},\label{monprod}\end{equation}
 $$M_{norm,p}=\diag(e^{2\pi i b_{1p}},e^{2\pi i b_{2p}})$$ 
 is the monodromy matrix of the formal normal form at $p$. Formula (\ref{monprod}) together with 
 unipotence of the Stokes matrices implies that 
 \begin{equation}\det M=\det M_{norm,p}=\exp(2\pi i\tr R).\label{dettr}\end{equation}

\begin{remark} \label{rk23} (see \cite[remark 2.3]{bibgl}). The canonical solutions $Y(z)$ 
of system (\ref{eqlin})  
on a Stokes sector $S_{j}^p$ (which are the columns of the fundamental matrix solution $X_p^j$) 
are numerated by the eigenvalues $\la_{sp}$  of its main term matrix 
at $p$ (or equivalently, by the eigenvalues $b_{sp}$ of the residue matrix of the corresponding 
formal normal form). The projectivized  canonical sectorial solutions $\Phi(z)=\frac{Y_2(z)}{Y_1(z)}$ 
(which take values in $\cp^1=\oc_\Phi$) are the unique  solutions of the  Riccati equation corresponding to (\ref{eqlin}) that are holomorphic on $S_{j}^p$ and 
extend continuously to $p$. The value of an extended projectivized canonical solution 
at $p$ is the tautological projection to $\cp^1=\oc$ of the eigenline $L_{sp}$ 
of the main term matrix corresponding to the eigenvalue $\la_{sp}$: the eigenvalue numerating the solution of linear system. 
\end{remark}
\begin{definition} \label{defms} \cite[definition 2.7]{bibgl}. 
Consider a linear system $\mcl$ of type (\ref{eqlin}). 
Fix a point $z_0\in\cc^*$ and two pairs of Stokes sectors 
($S_{0}^0$, $S_{1}^0$), ($S_{0}^{\infty}$, $S_1^{\infty}$) for the main term matrices $K$ at $0$ 
and $N$ at $\infty$ respectively. Fix two paths 
$\alpha_p$  in $\cc^*$ numerated by $p=0,\infty$, going from the point 
$z_0$ to a point  in $S_{0}^p$. Let $f_{1p}$, $f_{2p}$ be a canonical sectorial solution basis  
for the system $\mcl$ at $p$ in $S_0^p$. Consider the analytic extensions of the basic functions $f_{kp}$ 
to the point $z_0$ along the paths $\alpha_p^{-1}$. Let $\pi:\cc^2\setminus\{0\}\to\cp^1$ 
denote the tautological projection. Set $\Phi:=\frac{Y_2}{Y_1}$, 
\begin{equation} q_{kp}:=\pi(f_{kp}(z_0))\in\cp^1=\oc_{\Phi}.\label{qjp}\end{equation}
Let $M$ denote the monodromy operator of the system $\mcl$ acting on the local solution 
space at $z_0$ (identified with the space $\cc^2$ of initial conditions 
at $z_0$) by analytic extension along counterclockwise circuit around zero. The tuple 
\begin{equation}(q,M):=(q_{10},q_{20},q_{1\infty}, q_{2\infty}; M)
\label{msd}\end{equation}
taken up to the next equivalence is called the {\it monodromy--Stokes data} of the system 
$\mcl$. Namely, two tuples $(q,M),(q',M')\in(\cp^1)^4\times\gl_2(\cc)$ are called 
{\it equivalent\footnote{Here is an equivalent group-action definition. 
The group $\psl_2(\cc)$ acts on $\oc^4\times\gl_2(\cc)$ by action 
$h:q_{kp}\mapsto hq_{kp}$ on points in $\oc=\cp^1$ and conjugation $M\mapsto hMh^{-1}$ 
on matrices. The {\it monodromy-Stokes data} is the $\psl_2(\cc)$-orbit of a collection 
$(q,M)$ under this action.},} if there exists a linear operator $H\in\gl_2(\cc)$ whose projectivization 
sends $q_{kp}$ to $q'_{kp}$ and such that $H^{-1}\circ M'\circ H=M$. 
\end{definition}
\begin{remark} \cite[remark 2.8]{bibgl} 
The  monodromy--Stokes data of a system (\ref{eqlin}) depends only 
on the homotopy class of the pair of paths $(\alpha_0,\alpha_{\infty})$ 
in the space of pairs of paths in $\cc^*$ with a common (variable) starting point $z_0$ and with endpoints 
lying in given sectors $S_0^0$ and $S_0^\infty$ respectively.
 \end{remark}
Here and in what follows two linear systems on $\oc$ are said to be  {\it gauge equivalent,} if they are  ``constant gauge equivalent'':  obtained 
 one from the other by constant linear variable change $Y\mapsto HY$, $H\in\operatorname{GL}_2(\cc)$. 
 Gauge equivalence of holomorphic vector bundles with meromorphic connections is defined analogously, 
 with $H$ being a holomorphic  bundle isomorphism.

\begin{theorem} \label{gaugeq} \cite[theorem 2.11]{bibgl}. Two linear systems of type (\ref{eqlin}) are {\bf gauge 
equivalent,} if and only if they have the same formal normal 
forms at each singular point and the same monodromy-Stokes data (defined by a given collection of sectors 
and paths). In this case each linear automorphism of the fiber $\cc^2\times\{ z_0\}$ sending the monodromy-Stokes data of one system to that of the other system extends to a gauge equivalence of systems. Here both monodromy-Stokes   data correspond to the same sectors and path collections.
\end{theorem}
\begin{definition} A family of  linear systems (\ref{eqlin})  with irregular non-resonant singular points 
at zero and at infinity is {\it isomonodromic,} if the eigenvalues of 
the residue matrices of  formal normal forms at their singular 
points and the monodromy--Stokes data remain constant: independent on the 
parameter of the family. 
\end{definition}
\begin{remark} In the case under consideration (two singularities of Poincar\'e rank 1) 
the above definition given in \cite[definition 3.1]{bibgl} is known to be equivalent to the classical 
definition saying that a family of systems is isomonodromic, if the monodromy matrix (in appropriate canonical sectorial basis), Stokes matrices, the transition matrix between two canonical bases at zero and at infinity, 
and the above residue eigenvalues in normal forms  
are constant: see \cite[proposition 2.9]{bibgl}. 
It is well-known that if a family of systems (\ref{eqlin}) depends continuously on a parameter from a connected manifold, then constance of the monodromy-Stokes data automatically implies constance of the residue eigenvalues of the formal normal forms.
\end{remark}

\subsection{Degree of polynomial solution. Proof of Proposition \ref{poldeg}}

\begin{proposition} \label{polcan}
 A system (\ref{eqlin}) with Poincar\'e rank 1 irregular non-resonant singular points 
at $0$ and at $\infty$ has a vector polynomial solution, if and only if each its main term matrix at $0$, $\infty$ 
has a zero eigenvalue, the corresponding residue eigenvalue of the formal normal forms is  a nonnegative integer number  at $0$ 
and a  number $\ell\in\zz_{\geq0}$ at $\infty$, and the corresponding canonical sectorial solutions at zero and at infinity coincide 
(up to constant factor) and holomorphic in $\cc^*$. In this case  the polynomial solution is unique,  
coincides with the above canonical solution and has degree $\ell$. 
\end{proposition}
\begin{proof} At each singular point $p=0,\infty$ the corresponding main term matrix has at least one 
non-zero eigenvalue $\la$ (non-resonance). In each Stokes sector $S$ the corresponding canonical solution has type of 
an exponential function times a monomial (of some not necessary integer power) 
times a function holomorphic in $S$ that extends continuously to $p$. 
The Stokes sector $S$ contains a subsector $S'$ (lying on appropriate side from the Stokes ray contained in $S$) where the exponential function tends to infinity exponentially, as $z\to p$, $z\in S'$. Therefore, if 
 a polynomial solution exists, then its restriction to $S'$ has smaller growth. In other words it is asymptotically dominated by the above ``exponential'' canonical solution. But it is a well-known fact from the Stokes phenomena theory that if in some sector $S'$ one solution of a two-dimensional linear system (\ref{eqlin}) 
 dominates another solution, then the smaller, dominated solution is a canonical sectorial solution, and it is canonical in each Stokes sector containing $S'$. Namely this fact explains unipotence and triangularity of the 
 Stokes matrices.  Therefore, the polynomial solution  is canonical at both singular points $p=0,\infty$. 
 It corresponds to zero eigenvalue of the main term matrices, being meromorphic, and the corresponding 
 residue eigenvalues of the formal normal forms are clearly a 
 nonnegative integer number  at $0$ and a number 
 $\ell\in\zz_{\geq0}$ at infinity (equal to the degree of solution). Conversely, 
 if there exists a solution holomorphic on $\cc^*$ 
 that is canonical at each singular point and corresponds to zero eigenvalue of both  
 main term matrices and the corresponding residue eigenvalues of the formal normal form 
 are equal to nonnegative integer numbers $\nu$ at zero and $\ell$ at infinity, then the solution is polynomial of degree $\ell$. Uniqueness of polynomial solution follows from the above exponential growth and domination argument. Proposition \ref{polcan} is proved.
 \end{proof} 
 
Proposition \ref{poldeg}  follows immediately from Proposition \ref{polcan}.

 \subsection{Symmetric systems of class S and their monodromy--Stokes data}
 For a $2\times2$-matrix $K$ let $K^{tt}$ denote the matrix obtained from $K$ by 
 conjugation by the transposition matrix $\left(\begin{matrix} 0 & 1\\ 1 & 0\end{matrix}\right)$, i.e., obtained from $K$ by transposition of diagonal terms and transposition 
 of off-diagonal terms.
  \begin{definition} 
 A linear system is {\it of class $S$}, if it has the type 
 \begin{equation}Y'=\left(\frac K{z^2}+\frac R{z}-K^{tt}\right)Y, \ \ R_{12}=-R_{21}.\label{types}
 \end{equation}
 \end{definition}
 \begin{example} \label{exmchoy} Every system (\ref{mchoy*}) is of type $S$.\end{example}

\begin{proposition} \label{pros}
 A linear system (\ref{eqlin}) has type $S$, if and only if it admits a symmetry 
of the type 
\begin{equation}
\mathbb I:(Y_1(z), Y_2(z))\mapsto z^\ell(Y_2(z^{-1}), Y_1(z^{-1})).
\label{mathbi}\end{equation}
In this case one has 
\begin{equation}\ell=\tr R.\label{tracekr}\end{equation}
\end{proposition}
\begin{proof} The transformation $(Y_1,Y_2,z)\mapsto (Y_2,Y_1,z^{-1})$ sends graphs of solutions of a system (\ref{eqlin}) with matrices $K$, $N$, $R$ 
 to graphs of solutions of the new system of similar kind with 
the new matrices 
\begin{equation} \wh K=-PNP^{-1}, \  \wh N=-PKP^{-1}, \ \wh R=-PRP^{-1}; \ \ \ P:=\left(\begin{matrix} 0 & 1\\
1 & 0\end{matrix}\right).\label{conjperm}\end{equation} 
Multiplication by $z^\ell$ adds $\ell Id$ to $R$. 
The statement saying that $\mathbb I$ is a symmetry of the initial equation (\ref{eqlin}) 
is equivalent to the equality that the new 
matrices thus obtained are equal respectively to $K$, $N$ and $R$: 
\begin{equation} K=-PNP^{-1}=-N^{tt}, \ R=-PRP^{-1}+\ell Id=-R^{tt}+\ell Id.
\label{eqkn}\end{equation}
The last equation is equivalent to the fact that the off-diagonal elements of the matrix $R$ 
are opposite and $\ell=\tr R$.  This proves the proposition.
\end{proof}

 Recall that the canonical basic solutions of a system (\ref{eqlin}) at a point $p\in\{0,\infty\}$ 
 in a Stokes sector $S_j^p$ are  numerated by the eigenvalues $\la_{1p}$, $\la_{2p}$ of the corresponding main term matrix $K$ (respectively, $N$). Or equivalently, by the 
 eigenvalues $b_{1p}$, $b_{2p}$ of the residue matrix\footnote{The main term matrix eigenvalues  are always distinct  (the singularity in question in non-resonant), and hence, numerate the canonical basic solutions. But the residue matrix eigenvalues may coincide. One can say that they numerate the canonical basic solutions, only if they are distinct.}  of the corresponding formal normal form 
 (\ref{nform}), see Remark \ref{rk23}.  
  \begin{definition} \label{dsyms} Let a linear system (\ref{eqlin}) have irregular non-resonant singular points  of Poincar\'e rank 1 at the origin and at infinity.  Let $(q;M)=(q_{10}, q_{20}, q_{1\infty},q_{2\infty};M)$ 
  be its monodromy-Stokes data defined by the base point $z_0=1$, sectors $S_0^0$ and $S_0^{\infty}$ that are permuted by the involution $z\mapsto z^{-1}$, and  paths $\alpha_0$, $\alpha_\infty$ going from $z_0$ to $S_0^0$ and $S_0^\infty$ respectively and permuted by the same involution. (The latter pairs of sectors and paths  will be called {\it symmetric.}) We say that the 
  monodromy-Stokes data $(q,M)$ is {\it (1,1)-symmetric}, if there exists a linear involution 
  $\sigma:\cc^2\to\cc^2$ with eigenvalues $1$, $-1$ (i.e., conjugated to the above 
  permutation matrix $P$) 
   that conjugates the monodromy transformation with a scalar multiple of its inverse, and such that 
  the  projectivization of the involution $\sigma$ permutes $q_{10}$ and $q_{2\infty}$ and  permutes $q_{20}$ and 
  $q_{1\infty}$. In this case
  $$\sigma M\sigma^{-1}=\pm\exp(2\pi i \tr R) M^{-1},$$
  where $R$ is the residue matrix of the system (\ref{eqlin}), which follows from (\ref{dettr}). 
  \end{definition}
  \begin{remark} \label{11sym} Let in a system (\ref{eqlin}) the differences $\la_{20}-\la_{10}$, 
  $\la_{2\infty}-\la_{1\infty}$ of the eigenvalues in 
  the main term matrices $K$ and $N$ be real-proportional. (This is true, e.g., for systems of 
  type S, where the eigenvalue collections differ by sign.) Then the Stokes 
  ray collections at the origin and at infinity are permuted by the involution $z\mapsto z^{-1}$. This 
  implies that the Stokes sector collections at $0$ and at $\infty$ can be chosen so that 
  $S_j^0$ and $S_j^{\infty}$ be also permuted by the same involution for every $j=0,1$. 
  Therefore, one can define the monodromy-Stokes data by paths $\alpha_0$ and $\alpha_\infty$ 
  permuted by the same involution, and the notion of (1,1)-symmetric monodromy-Stokes data is well-defined. 
  \end{remark}

    \begin{theorem} \label{thgauge} A linear system (\ref{eqlin}) is gauge equivalent to a linear system of type $S$, if and only if 
  \begin{equation} (\la_{1\infty},\la_{2\infty}; b_{1\infty}, b_{2\infty}) = (-\la_{20}, -\la_{10}; \ 
    b_{10}, b_{20}),\label{labeq}\end{equation}
     and its monodromy-Stokes data $(q_{10},q_{20}, q_{1\infty},q_{2\infty}; M)$ 
  (numerated by the eigenvalues $\la_{10}$, $\la_{20}$, $\la_{1\infty}$, $\la_{2\infty}$ and constructed from some given symmetric pairs of sectors and paths) is (1,1)-symmetric. (In this case it is 
  (1,1)-symmetric for every symmetric pairs of sectors and paths.) By definition, equality (\ref{labeq}) 
  means that its formal normal forms at the origin and at the infinity are respectively
  $$Y'=\left(\frac{\diag(\la_{10},\la_{20})}{z^2}+\frac{\diag(b_{10},b_{20})}{z}\right)Y \ \ \text{ at } \ 0,$$
  $$Y'=\left(-\diag(\la_{20},\la_{10})+\frac{\diag(b_{10},b_{20})}{z}\right)Y \ \ \text{ at } \ \
  \infty.$$
  \end{theorem}
  \begin{proof}  Let us prove the statements of the theorem for a system of type S.  First let us prove (\ref{labeq}). 
  Equality $\{\la_{1\infty},\la_{2\infty}\}=\{-\la_{20},-\la_{10}\}$ follows from definition. Let us numerate the 
  latter eigenvalues so that $\la_{1\infty}=-\la_{20}$, $\la_{2\infty}=-\la_{10}$.  Let $H$ be the matrix diagonalizing $K$: $H^{-1}KH=\diag(\la_{10},\la_{20})$, and let 
  $P$ be the permutation matrix, see (\ref{conjperm}). One has $P^2=Id$. 
  The matrix $PHP$ diagonalizes the matrix $N=-K^{tt}=-PKP$: 
  $$P^{-1}H^{-1}P^{-1}NPHP=-P^{-1}H^{-1}KHP=-\diag(\la_{20},\la_{10})=\diag(\la_{1\infty},\la_{2\infty}).$$
One has 
  $$\diag(b_{10}, b_{20})= \text{ the diagonal part of } H^{-1}RH,$$
  $$\diag(b_{1\infty},b_{2\infty})=\text{ the diagonal part of } 
  P^{-1}H^{-1}P^{-1}RPHP.$$
  One has $P^{-1}RP=-R+(\tr R)Id$, by definition and since $R_{12}=-R_{21}$. Therefore, 
  $$\diag(b_{1\infty},b_{2\infty})=-(\text{the diagonal part of } \  P^{-1}H^{-1}RHP)+(\tr R) Id
  $$
  $$=-\diag(b_{20},b_{10})+(\tr R)Id=\diag(b_{10},b_{20}).$$

  Let us prove symmetry of the monodromy-Stokes data. First let us  show that 
  \begin{equation}PMP^{-1}=e^{2\pi i\ell}M^{-1}, \ \ \ell=\tr R.\label{conjmon}\end{equation}
  Indeed,  the monodromy operator $M$ 
  acts on the space of initial conditions in $\cc^2=\cc^2\times\{1\}$ at $z_0=1$ so that 
  an initial condition $Y^0$ is sent to the analytic extension of the corresponding germ of solution 
  to $1$ along a counterclockwise circuit around the origin. Let us write its matrix in the standard 
  basis $(1,0)$, $(0,1)$.  The transformation (\ref{mathbi}) is a symmetry of system 
  (\ref{eqlin}) (Proposition \ref{pros}). 
  It is the composition of the three following operations:
  
  (i) the space variable change $Y\mapsto PY$ (constant gauge transformation), 
  it conjugates the monodromy by the involution $P$; 
  
  (ii) the time variable change $z\mapsto z^{-1}$, it changes the monodromy to its inverse; 
  
  (iii) multiplication by the function $z^\ell$, it multiplies the monodromy by the 
  scalar factor $e^{2\pi i\ell}$.
  
  Therefore, the transformation $\mathbb I$ given by (\ref{mathbi}) on the space of vector functions sends solutions 
  of the initial system (with monodromy matrix $M$) to solutions of the new system whose monodromy matrix is equal to $e^{2\pi i\ell}PM^{-1}P^{-1}$. The latter matrix coincides with 
  $M$, since both systems coincide ($\mathbb I$ is a symmetry). This proves (\ref{conjmon}). 
  
  Let us now prove that the transformation $\mathbb I$ sends canonical solutions to canonical solutions.
  For every $p=0,\infty$ the transformation (i) sends canonical solutions at $p$ of the initial system 
  to canonical solutions of its image at the  same point $p$, and the numeration by  eigenvalues 
  $\la_{jp}$ is preserved. The transformation (ii) sends a canonical solution at $0$ corresponding 
  to the eigenvalue $\la_{j0}$ in a Stokes sector $S^0$ to the canonical solution of the transformed system at $\infty$ in the Stokes sector $S^\infty$ that is the image of the sector $S^0$ under the involution 
  $z\mapsto z^{-1}$, 
  and the corresponding eigenvalue of the main term matrix at infinity is equal to $-\la_{j0}=\la_{2-j,\infty}$. 
  This follows from Remark \ref{rk23}. 
This implies that the projectivization of 
  the operator $P$ is an involution $\oc_\Phi\to\oc_\Phi$ permuting $q_{10}$ and $q_{2\infty}$ 
  and permuting $q_{20}$ and $q_{1\infty}$. This together with (\ref{conjmon}) implies symmetry 
  of the monodromy-Stokes data. 
  
  Let us now prove the converse. Let for a given system (\ref{eqlin}) 
  the eigenvalue equalities (\ref{labeq}) hold, 
  and let the monodromy-Stokes data defined by some symmetric pairs of sectors and paths 
  be (1,1)-symmetric. Let us prove that (\ref{eqlin}) has type 
  S. Let $\sigma:\cc^2\to\cc^2$ be the 
  corresponding linear involution making their symmetry. Applying a gauge transformation we can and will consider that $\sigma=P$. Set $\ell=\tr R$. Then transformation (\ref{mathbi}) 
  sends the initial system (\ref{eqlin}) to a system of similar kind, let us denote it by 
  (\ref{eqlin})',  with the same monodromy-Stokes data up to multiplication of the monodromy by $\pm1$. 
  This follows from the above discussion on action of 
  the operations (i)--(iii) and from the formula in Definition \ref{dsyms}. 
  At each singular point the eigenvalues of the main term 
  matrices in both systems 
  are the same, and the eigenvalues of the residue matrices of the formal normal form 
are the same. In both systems they numerate the monodromy-Stokes data in the same way. 
This follows  by construction and (\ref{labeq}).  The monodromies of the systems 
coincide, by their coincidence up to sign, coincidence of the 
eigenvalues of the residue matrices (and hence, the monodromies) of their formal 
normal forms, formula (\ref{monprod}) and unipotence and opposite triangularity of the Stokes matrices. 
Finally, systems (\ref{eqlin}) and 
  (\ref{eqlin})' have the same formal normal forms and monodromy-Stokes data. Hence, they are 
  gauge equivalent: obtained one from the other by constant linear change 
  $Y\mapsto HY$, $H\in GL_2(\cc)$ (Theorem \ref{gaugeq}). 
  
   We choose coordinates $(Y_1,Y_2)$ in $\cc^2$, let $\Phi=\frac{Y_2}{Y_1}$ denote the corresponding 
   coordinate on $\cp^1=\oc$, so that the representatives $(q_{10}, q_{20}, q_{1\infty}, 
  q_{2\infty}; M)$ of the monodromy-Stokes data of systems (\ref{eqlin}) and (\ref{eqlin})' coincide. 
  Let us show that $H$ is scalar: a multiple of the identity.
  The action of the operator $H$ on the space of initial conditions 
  $\cc^2=\cc^2\times\{ 1\}$ preserves the monodromy-Stokes data. This implies that its 
  projectivization preserves the points $q_{jp}$ and their images under the projectivized 
  monodromy: the monodromy of the Riccati equation corresponding to (\ref{eqlin}). 
  Let the total 
  number of distinct points among the points $q_{jp}$ and their latter images  be greater than two. Then the projectivized operator $H$ preserves at least three distinct points in $\oc$, and hence, is identity. Thus $H$ is scalar, and the above variable change multiplies $Y$ by a constant. This 
  scalar change does not change the underlying linear system. Hence, systems (\ref{eqlin}) 
  and (\ref{eqlin})' coincide, and thus, the transformation $\mathbb I$, see (\ref{mathbi}), 
  is a symmetry of system (\ref{eqlin}). Therefore, (\ref{eqlin}) is a system of type S (Proposition  
  \ref{pros}). 
  
  Let us consider the opposite case: the above number of distinct points 
  is  at most two. For every given $p=0,\infty$ the points $q_{1p}$, $q_{2p}$ are distinct, being 
  projectivized initial conditions of two linearly independent basic solutions of (\ref{eqlin}). 
  This implies that $\{ q_{10},q_{20}\}=\{ q_{1\infty},q_{2\infty}\}$, and the projectivized 
  monodromy either preserves each $q_{jp}$, or permutes $q_{1p}$ and $q_{2p}$. The second case is impossible. Indeed, otherwise, if $q_{10}$ and $q_{20}$ were permuted by the projectivized monodromy, then in the canonical basis at $0$ the monodromy matrix of system (\ref{eqlin}) 
  would have zero diagonal terms. But at least one of them should be nonzero and equal to 
  an eigenvalue $e^{2\pi i b_{j0}}$ of the monodromy of the formal normal form (follows 
  from (\ref{monprod}), opposite triangularity and unipotence of the Stokes matrices), -- 
  a contradiction. In the 
  first case, when the projectivized monodromy preserves each $q_{jp}$, 
   the monodromy of linear system is diagonal in the canonical sectorial basis at $p$. Hence, 
  the corresponding Stokes operators are trivial, by their opposite triangularity and 
  (\ref{monprod}) (as in \cite[proof of lemma 3.3]{4}). And this holds at both points 
  $p=0,\infty$. This together with Remark \ref{rk23} implies that the projectivized canonical solutions $g_{jp}(z)=\pi\circ f_{jp}(z)$, i.e., the solutions of 
  the Riccati equation with initial conditions $q_{j0}$ at $z_0=1$, are meromorphic on the 
  whole Riemann sphere, and their values at each point $p=0,\infty$ are the two distinct 
  eigenlines of the corresponding main term matrix ($K$ or $N$). This implies that the graphs 
  $\{\Phi=g_{j0}(z)\}$, $j=1,2$, are disjoint rational curves in $\oc_\Phi\times\oc_z$ parametrized by $\oc_z$ via projection. But they may be disjoint only if $g_{j0}\equiv const$. Indeed, otherwise, 
  if some $g_{j0}$ is non-constant,  then the intersection index of graphs would be positive, by 
   K\"unneth Formula for product homology and bi-degree argument, see the proof of the same statement 
   at the end of \cite[proof of proposition 4.8]{bibgl}, -- a contradiction. Hence, 
   $g_{10}(z)$ and $g_{20}(z)$ are  constant. Applying a gauge transformation, we can and will 
   consider that $g_{10}\equiv 0$, $g_{2,0}\equiv\infty$. Then the matrices in linear system 
   (\ref{eqlin})  are diagonal, as in loc. cit., and the system is obviously of type S. Theorem 
   \ref{thgauge} is proved.
   \end{proof}
   
   \begin{lemma} \label{smch} A system of type S where the main term matrix $K$ has distinct eigenvalues 
   (non-resonance condition) is gauge equivalent to a system of type (\ref{mchoy*}), if and only if 
   in (\ref{labeq}) one has $\la_{20}=b_{20}=0$ and the following {\bf line non-coincidence condition} holds:

(*) the eigenlines $L_{10}$, $L_{2\infty}$ of the main term matrices at $0$  and  
   at  $\infty$ respectively corresponding to non-zero eigenvalues  $\la_{10}$ and $\la_{2\infty}=-\la_{10}$  
   are distinct.\end{lemma}
   \begin{proof} Each system (\ref{mchoy*}) obviously satisfies equality $\la_{20}=b_{20}=0$ and 
   condition (*), see formula (\ref{ellll}) for analogous system with $\ell$ replaced by $-\ell$. Let us prove the converse. Let $\la_{20}=b_{20}=0$ and condition (*) hold. 
   The permutation operator $P$ permutes the eigenlines 
   $L_{10}$ and $L_{2\infty}$, see (\ref{eqkn}). 
 Applying a gauge transformation $Y\mapsto HY$ with $H$ commuting with $P$ 
 preserves the class S of linear systems, since it does not change  relations (\ref{eqkn}).  
 We can apply $H$ commuting with $P$ 
 that sends the eigenline $L_{10}$ to the first coordinate line, and then $L_{2\infty}$ is sent to 
 the second coordinate line, since they are distinct, by (*). This yields another system of class 
 S with upper-triangular matrix $K$  and lower-triangular matrix $N=-K^{tt}$. 
 The new system has type (\ref{mchoy*}) up to addition of a scalar matrix 
 $\nu Id$ to the residue matrix $R$, by (\ref{eqkn}).   In our case $\nu=0$, 
 since $b_{20}=0$ for both systems. Thus, the system under consideration is of type (1.32).  This proves the lemma.
 \end{proof}

   \begin{corollary} \label{lmchoy} A system (\ref{eqlin})  is gauge equivalent to a system of type 
   (\ref{mchoy*}), if and only if the following conditions hold:
   
 a) equalities (\ref{labeq}) hold, and $\la_{20}=\la_{1\infty}=0$, $b_{20}=b_{2\infty}=0$;
 
 b) the monodromy-Stokes data constructed from some given symmetric pairs of sectors and paths 
  is (1,1)-symmetric;
 
 c)  the above line non-coincidence condition (*) holds.
 
 In this case  monodromy-Stokes data defined by arbitrary 
 symmetric pairs of sectors and paths are (1,1)-symmetric. 
 \end{corollary}
 \begin{proof} Each  system (\ref{mchoy*}) obviously satisfies statements a) and c), and 
 statement b) follows by Theorem \ref{thgauge}. Conversely, let a system (\ref{eqlin}) satisfy 
 conditions a)--c). Then it is gauge equivalent to a system of class S, by symmetry of 
 monodromy-Stokes data and Theorem \ref{thgauge}. The latter system is gauge 
 equivalent to a system of type (\ref{mchoy*}), by Lemma \ref{smch}. The corollary is proved.
  \end{proof}

\subsection{Polynomial solutions and determinantal surfaces. Proof of Theorem \ref{tmatr}.} 
Everywhere below we consider that $s\neq0$.
\begin{proposition} \label{pinv1} Let $\ell\in\zz_{\geq0}$. The transformation (\ref{mathbi}) 
is an involution sending each vector polynomial solution $Y(z)$  of system (\ref{mchoy*}) (if any) to $\pm Y(z)$.
\end{proposition}
\begin{proof} The fact that (\ref{mathbi}) is an involution is obvious. It is a symmetry of system (\ref{mchoy*}), 
by Proposition \ref{pros} and since (\ref{mchoy*}) is a system of type S, see Example \ref{exmchoy}. 
If a  polynomial solution exists, then it is unique (up to constant factor) and has degree $\ell$ 
(Proposition \ref{polcan}). The transformation (\ref{mathbi}) preserves the space of vector polynomials of degree 
at most $\ell$ and acts there as a linear involution sending a solution of (\ref{mchoy*}) to a solution. 
Therefore, the unique polynomial solution is its eigenvector with eigenvalue $\pm1$.
\end{proof}

\begin{proposition} \label{pinv2} A vector polynomial $Y(z)=(Y_1(z),Y_2(z))$ is a solution of linear system 
(\ref{mchoyn}), if and only if it satisfies the system of two following equations with at least one 
choice of sign $\pm$:
\begin{equation}  (Y_1(z), Y_2(z))=\pm z^{\ell}(Y_2(z^{-1}),Y_1(z^{-1}))\label{wsymm}
\end{equation}
\begin{equation} Y_2'(z)=\pm(\frac a{2z}+s\chi)z^\ell Y_2(z^{-1})+(\frac{\chi a}z+\frac s2)Y_2(z)
\label{diffred}\end{equation}
\end{proposition}

\begin{proof} Let a vector polynomial $Y(z)$ be a solution of system (\ref{mchoyn}). Then 
equation (\ref{wsymm}) holds for some choice of sign $\pm$, by Proposition \ref{pinv1}. 
Substituting thus found $Y_1(z)=\pm z^{\ell}Y_2(z^{-1})$ to the second line of (\ref{mchoyn}) 
yields (\ref{diffred}). Conversely, let a vector polynomial $Y(z)$ be a solution of 
system of equations (\ref{wsymm}) and (\ref{diffred}). Then the second line of system 
(\ref{mchoyn}) holds, as in the above argument. It remains to show that its first line holds as well. 
Differentiating $Y_1(z)=\pm z^\ell Y_2(z^{-1})$ and substituting (\ref{diffred}) yields 
\begin{equation}Y_1'(z)=\frac{\ell}z Y_1(z)\mp z^{\ell-2}Y_2'(z^{-1}),\label{w1diff}\end{equation}
$$\mp z^{\ell-2}Y_2'(z^{-1})=-z^{\ell-2}(\frac{az}{2}+s\chi)z^{-\ell} Y_2(z)
\mp z^{\ell-2}(\chi a z+\frac s2)Y_2(z^{-1}).$$
Substituting the equality $Y_2(z^{-1})=\pm z^{-\ell} Y_1(z)$ to the latter right-hand side and substituting everything to  (\ref{w1diff}) yields the first line in (\ref{mchoyn}): 
$$Y_1'(z)=(\frac{\ell-\chi a}z-\frac s{2z^2})Y_1(z)-(\frac a{2z}+\frac{s\chi}{z^2})Y_2(z).$$
\end{proof}

\begin{proof} {\bf of Theorem \ref{tmatr}.} It suffices to show that existence of polynomial solution 
of system of equations (\ref{wsymm}) and (\ref{diffred}) for a given choice of sign $\pm$ is 
equivalent to the inclusion $(\chi,a,s)\in S_{\ell,\pm}$. The space of polynomial solutions of 
equation (\ref{wsymm}) is isomorphic to the space of polynomials $Y_2(z)$ of degree 
at most $\ell$ via the correspondence $Y_1(z)=\pm z^{\ell}Y_2(z^{-1})$. A given polynomial 
$Y_2(z)=\sum_{j=0}^\ell c_jz^j$ of degree at most $\ell$ is a solution of (\ref{diffred}), if and only if its coefficients satisfy the following relations, corresponding to equality of degree $j$ terms in  (\ref{diffred}) for $j=0,\dots,\ell$:
\begin{equation}\pm(\frac a2 c_{\ell-(j+1)}+ s\chi c_{\ell-j})+(\chi a-(j+1)) c_{j+1}+\frac s2 c_{j}=0.
\label{relj}\end{equation}
Here we set $c_k=0$ for $k\notin\{0,\dots,\ell\}$. System of relations (\ref{relj}) is equivalent to 
the statement that the vector of coefficients $(c_0,\dots, c_\ell)$ lies in the kernel of the matrix 
$G_{1,\ell}\pm G_{2,\ell}$. Therefore, existence of polynomial solution of 
system of equations (\ref{wsymm}), (\ref{diffred}) 
is equivalent to the statement that the latter matrix has zero determinant, i.e., 
$(\chi,a,s)\in S_{\ell,\pm}$. This together with Proposition \ref{pinv2} implies the statement of 
Theorem \ref{tmatr}. 
 \end{proof}
 \begin{remark} The above proof of Theorem \ref{tmatr} remains valid for $\ell=0$ and yields 
 a proof of Proposition \ref{pel0} for $s\neq0$. Indeed,  for $\ell=0$ equation (\ref{relj}) becomes 
 $\pm s\chi c_0+\frac s2c_0=0$ and has a non-zero solution, if and only if $\chi=\mp\frac12$. 
 \end{remark}

 \subsection{Systems with polynomial solutions. Parabolicity of the monodromy-Stokes data}
 \begin{definition} \label{qdiff} We deal with an {\it abstract monodromy-Stokes data:} 
 a collection $(q_{10}, q_{20}, q_{1\infty}, q_{2\infty}; M)$, where $q_{jp}\in\oc_\Phi=\cp^1_{[Y_1:Y_2]}$, $\Phi=\frac{Y_2}{Y_1}$, 
 $q_{10}\neq q_{20}$, $q_{1\infty}\neq q_{2\infty}$, and $M\in GL_2(\cc)$ is a non-degenerate 
  linear operator acting on the space $\cc^2_{Y_1,Y_2}$, called the {\it monodromy operator}. The above collection is taken up to equivalence from the definition of the monodromy-Stokes data. 
 We say that an abstract  monodromy-Stokes data  is {\it parabolic,} if  $q_{20}=q_{1\infty}$, and  the  projectivized monodromy operator is either a parabolic transformation $\oc\to\oc$ that fixes only the point $q_{20}$, or the identity. Set $q_{10}'$ to 
 be the image of the point $q_{10}$ under the projectivized monodromy transformation. 
 A parabolic monodromy-Stokes data is called {\it non-degenerate,} if among the four points 
 $q_{10}'$, $q_{10}$, $q_{20}$, $q_{2\infty}$ at least three points are distinct. 
  \end{definition}
  \begin{remark} \label{aut11} Consider a monodromy-Stokes data of linear system 
  (\ref{eqlin})   defined by symmetric pairs of sectors and paths. 
  If the monodromy-Stokes data is parabolic, then it is 
  automatically (1,1)-symmetric. This follows from the fact that every conformal involution $\sigma:\oc\to\oc$ 
  fixing $q_{20}=q_{1\infty}$ (e.g., the one  permuting $q_{10}$ and $q_{2\infty}$) conjugates each parabolic 
  transformation fixing $q_{20}$ (e.g., the 
  projectivized monodromy) with its inverse. Namely, if we choose the coordinate $w$ on $\oc$ so that 
  $q_{20}=\infty$ and the origin is the second fixed point of the  involution $\sigma$, then 
  $\sigma(w)=-w$, the projectivized monodromy is a translation $w\mapsto w+u$, and hence,  
  $\sigma$ conjugates it with its inverse $w\mapsto w-u$. 
  \end{remark}

  Here we prove the following theorem.
  
   \begin{theorem} \label{thpar} A system (\ref{eqlin}) with irregular non-resonant singular points of Poincar\'e rank 1 at $0$ and 
 at $\infty$ is gauge equivalent to a system of type (\ref{mchoyn}) having a non-constant polynomial solution, if and only if it satisfies the following conditions:
 
 a) $(\la_{1\infty},\la_{2\infty}; b_{1\infty}, b_{2\infty}) = (-\la_{20}, -\la_{10}; \ 
    b_{10}, b_{20})$,  $\la_{20}=\la_{1\infty}=0$, 
  $b_{20}=b_{2\infty}=0$.
  
  b) $\ell:=b_{10}=b_{1\infty}\in\nn$.
 
 c) The monodromy-Stokes data defined by  sectors and paths permuted by  
  the involution $z\mapsto\frac1z$ is parabolic non-degenerate.
 
 d) The above line non-coincidence condition (*) holds.
 
 In this case Statement c) holds for monodromy--Stokes data defined by arbitrary symmetric pairs 
 of sectors and paths. 
 \end{theorem}
 
 \begin{proof} 
 First let us prove that each system (\ref{mchoyn}) with a non-constant polynomial solution 
 satisfies statements a)--d). Statements a) and d) hold for each system (\ref{mchoyn}). Statement b) follows 
 from Proposition \ref{poldeg}. 
 Let us now prove parabolicity of the monodromy-Stokes data. 
 The monodromy operator has unit determinant, since it is the product of unipotent matrices and 
 the monodromy of the formal normal form, i.e., the formal monodromy, see (\ref{monprod}), and the formal monodromy is equal to $\diag(e^{2\pi i\ell}, 1)=Id$, since $\ell\in\nn$. Hence, the monodromy 
 is unipotent: it has unit determinant and at least one unit eigenvalue, which corresponds  
 to the polynomial solution (fixed by the monodromy).  The value at $z=1$ of the 
 projectivized polynomial solution is equal to $q_{20}=q_{1\infty}$, and it is a fixed point of the projectivized monodromy, by the above similar statement on the monodromy.  This together with 
  unipotence of the monodromy 
 implies  parabolicity of the monodromy-Stokes data. Let us prove its non-degeneracy.

 Suppose the contrary: among the points $q_{jp}$, $q_{10}'$ at most two are distinct. 
 Then $q_{10}=q_{2\infty}=q_{10}'$, since $q_{10}\neq q_{20}=q_{1\infty}$, $q_{1\infty}\neq q_{2\infty}$,  $q_{20}$ is a fixed  point of the projectivized monodromy transformation, and thus, the image $q_{10}'$ of the point $q_{10}$ under the projectivized monodromy is distinct from $q_{20}$.  Therefore,  
 $q_{10}=q_{2\infty}$, $q_{20}=q_{1\infty}$, and they are fixed by the projectivized monodromy. 
 Thus, the corresponding linear system is  gauge equivalent to a diagonal system, as at the end of the proof of Theorem \ref{thgauge}.  Hence, $\chi=a=0$, and thus, 
system (\ref{mchoyn}) is diagonal. But then it has no non-constant polynomial solution, since 
it is the direct sum of one-dimensional equations $Y_1'=(-\frac s{2z^2}+\frac{\ell-\chi a}z)Y_1$, 
 $Y_2'=(\frac{\chi a}z+\frac s2)Y_2$, and none of them has a non-constant polynomial  solutuon, --  a contradiction. 
  
 Let us prove the converse: assyming conditions a)-d), let us show 
 that the system (\ref{eqlin}) in question is gauge 
 equivalent to a system (\ref{mchoyn}) having a non-constant polynomial solution. We already know 
 that the system is gauge equivalent to a system (\ref{mchoyn}), by Corollary \ref{lmchoy}: the monodromy-Stokes data is (1,1)-symmetric, see Remark \ref{aut11}. Thus, we 
 consider that our system is of type (\ref{mchoyn}). Its 
projectivized canonical solutions corresponding to the eigenvalues $\la_{20}=\la_{1\infty}=0$ at 
the origin and at infinity
paste together to the projectivization of just one solution $Y(z)$ of linear system, 
since $q_{20}=q_{1\infty}$ (parabolicity). The projectivized solution $Y(z)$ is fixed by the projectivized monodromy, since the latter fixes $q_{20}$  (parabolicity). Hence,  $Y(z)$ is a canonical solution of linear system, and the monodromy multiplies $Y(z)$ by a constant factor. The formal residue eigenvalues at $0$ and at $\infty$ corresponding to $Y(z)$ are respectively $b_{20}=0$ and $b_{1\infty}=\ell$, by definition, and $\ell\in\nn$ by assumption. 
This together with (\ref{monprod}) implies that in the canonical sectorial solution basis at the origin  
the monodromy is triangular and $Y(z)$ is its eigenvector with unit eigenvalue. Thus, 
$Y(z)$ is holomorphic on all of $\cc^*$. Recall that the canonical solution $Y(z)$ corresponds 
to zero main term  eigenvalues at $0$ and $\infty$. This together with the above formal 
residue eigenvalue statement imply that $Y(z)$  extends holomorphically to the origin and has growth of order $z^\ell$ (times a non-zero constant vector), 
as $z\to\infty$. Hence, it  is a nonconstant vector polynomial solution. Theorem \ref{thpar} is proved.
\end{proof}

In what follows we use the following general property of bundles with connections having 
parabolic monodromy-Stokes data.

\begin{proposition} \label{pbng} Let a holomorphic vector bundle $\mce$ on the Riemann sphere $\oc_z$ 
be equipped with a meromorphic 
connection with irregular non-resonant singular points at $0$ and at $\infty$ with formal 
normal forms 
$$\dot Y=\left(\frac{\diag(-\frac s2, 0)}{z^2}+\frac{\diag(\ell,0)}z\right)Y \ \text{ at } 0,$$
$$\dot Y=\left(\diag(0,\frac s2)+\frac{\diag(\ell,0)}z\right)Y \ \text{ at } \infty.$$
Let $S_0^0$ and $S_0^{\infty}$ be Stokes sectors at $0$ and at $\infty$ respectively that are 
permuted by the involution $z\mapsto z^{-1}$. 
Let $\alpha_0$ and $\alpha_\infty$ be paths from 1 to $S_0^0$ and $S_0^\infty$ respectively 
that are also permuted by the same involution. 

1) The connection has a meromorphic solution (i.e., a flat section), if and only if $\ell\in\zz$ and the monodromy-Stokes data\footnote{The notion of monodromy-Stokes data is well-defined for 
connections in not necessarily trivial bundles over the Riemann sphere that have irregular non-resonant singular points of Poincar\'e rank 1 at $0$ and at $\infty$. In our case, when the eigenvalue 
collections of the main term matrices differ by sign, the notion of parabolic monodromy-Stokes data is also well-defined.} $(q_{10}, q_{20}, q_{1\infty},q_{2\infty})$ defined by the paths $\alpha_{p}$, $p=0,\infty$ and indexed by the main term eigenvalues $-\frac s2$, $0$, $0$, $\frac s2$, is parabolic. In this case the 
meromorphic solution is unique up to constant factor and is holomorphic on $\cc$. 

2) Let a meromorphic solution exist. Then  the bundle with connection is diagonal, i.e., a direct sum of one-dimensional 
holomorphic vector bundles with connections, if and only if  
 the monodromy-Stokes data is degenerate.
\end{proposition}
\begin{proof} A meromorphic solution $f$ should coincide with a canonical solution 
at the origin and at infinity that corresponds to zero eigenvalue of the main term matrix: thus, these 
solutions are pasted together, and hence, $q_{20}=q_{1\infty}$. This together with the equality $b_{20}=0$ 
implies holomorphicity on $\cc$ and uniqueness. The corresponding eigenvalue of the residue matrix of the formal normal form should be integer, by meromorphicity. Hence, $\ell\in\zz$. Therefore, at each singular 
point $0$, $\infty$ the formal monodromy is trivial, and at least one of the Stokes matrices is also trivial, since the canonical solution $f$ is meromorphic. 
Hence, the monodromy of the connection is unipotent, by   (\ref{monprod}). This together the 
equality $q_{20}=q_{1\infty}$ implies parabolicity of the monodromy-Stokes data. The converse 
is proved analogously to the above proof of Theorem \ref{thpar}. Statement 1) is proved.

Let in the above assumptions the bundle with connection be diagonal. 
Then the solutions of the connections in one-dimensional 
subbundles are canonical solutions at both singular points: the origin and the infinity. 
This implies that the connection is analytically equivalent to its formal normal form, and its  
monodromy is trivial, as is that of the normal form. The 
projectivizations of the above solutions 
 are holomorphic sections of the projectivized ambient bundle on $\cc^*$. Their  
values at $z=1$ are the points $q_{jp}$  from the monodromy-Stokes data. This 
together with the previous statement implies that $q_{10}=q_{2\infty}$, $q_{20}=q_{1\infty}$. 
Hence, the monodromy-Stokes data is degenerate. The converse is proved analogously to the 
proof of Theorem \ref{thpar}.
\end{proof} 

\subsection{Local parametrization  by formal invariant $s$ and  cross-ratio $\mcr$}
\begin{definition} Let $(q_{10}, q_{20}=q_{1\infty},  q_{2\infty}; M)$ be an abstract 
non-degenerate 
parabolic monodromy-Stokes data. Its {\it characterizing cross-ratio} is the cross-ratio of the four points $q_{20}$, $q_{10}$, $q_{10}'$, $q_{2\infty}$, see Definition \ref{qdiff}: 
\begin{equation}\mcr:=\frac{(q_{20}-q_{10}')(q_{10}-q_{2\infty})}{(q_{20}-q_{2\infty})(q_{10}-q_{10}')}.\label{crat}\end{equation}
\end{definition}

\begin{proposition} \label{abmon} Each abstract non-degenerate parabolic monodromy-Stokes data is uniquely 
determined by its characterizing cross-ratio, up to multiplication of the monodromy operator 
by scalar factor.
\end{proposition}
\begin{proof} The cross-ratio $\mcr$ defines the four points $q_{10}$, $q_{10}'$, 
$q_{20}$, $q_{2\infty}$ on the Riemann sphere $\oc$ up to M\"obius transformation. Except may be for some of the cases, when some of them coincide. Let us check that this statement remains 
valid in case of coincidences as well. Indeed, $q_{10}, q_{2\infty}\neq q_{20}=q_{1\infty}$. One has $q_{10}'\neq q_{20}$: the
 point $q_{10}'$ is the image of the point $q_{10}\neq q_{20} $ under the projectivized monodromy, which  fixes the point $q_{20}$. Therefore, the only a priori possible coincidences are either 
 $q_{10}=q_{2\infty}$, or $q_{10}'=q_{2\infty}$, or $q_{10}'=q_{10}$: only one coincidence may take place, since the total number of distinct points among $q_{10}$, $q_{10}'$, $q_{20}$, 
 $q_{2\infty}$ is at least three (non-degeneracy). In the case, when $q_{10}=q_{2\infty}$, one has $\mcr=0$. 
 In the case, when $q_{10}'=q_{2\infty}$, one has $\mcr=1$. In the case, when 
 $q_{10}'=q_{10}$, one has $\mcr=\infty$. This yields three distinct values for $\mcr$ for the three 
 possible coincidences. 
 
 Let us first assume for simplicity that  $q_{10}\neq q_{2\infty}$. Then 
 applying a M\"obius transformation we can and will consider that $q_{20}=\infty$, 
$q_{10}=0$, $q_{2\infty}=1$. Then in the coordinate on $\cc=\oc\setminus\{ q_{20}\}$, which will be denoted by $w$,  the projectivized monodromy operator is a translation $w\mapsto w+u$, $u\in\cc$. One should have 
$q_{10}'=q_{10}+u=u$, since $q_{10}'$ is the image of the point $q_{10}=0$. Therefore, 
$$\mcr=\frac{q_{2\infty}-q_{10}}{q_{10}'-q_{10}}=\frac{1}u.$$
Thus, the correspondence $\mcr\mapsto u$ is bijective, and the projectivized monodromy 
can be restored from $\mcr$, as does $u$. The case, when $q_{10}=q_{2\infty}$, i.e., 
$\mcr=0$, is treated analogously. Namely, in this case $q_{10}'\neq q_{10}, q_{20}$ (non-degeneracy). Applying a M\"obius transformation, we consider that $q_{20}=q_{1\infty}=\infty$,  
$q_{10}=q_{2\infty}=0$ and $q_{10}'=1$ in the standard affine coordinate $w$ on $\cc=\oc\setminus\{ q_{20}\}$. Then the projectivized monodromy is the translation $w\mapsto w+1$. Hence, the value 
$\mcr=0$ determines the monodromy-Stokes data uniquely up to multiplication of the monodromy 
by scalar factor.  Proposition \ref{abmon} is proved.
\end{proof}

\begin{theorem} \label{threq} Two systems of type (\ref{mchoyn}) with $s\neq0$ 
having  polynomial solutions of degree $\ell\geq1$ coincide up to 
the parameter change $(\chi,a)\mapsto(-\chi,-a)$, if and only if the corresponding values of the paramer $s$ and the cross-ratio $\mcr$ (constructed for given symmetric pair of homotopic classes of paths 
$\alpha_{jp}$) are the same for both systems. 
\end{theorem}

\begin{proof} If two systems of type (\ref{mchoyn}) are gauge equivalent, then the gauge equivalence is given by a matrix of the type $\diag(1,\pm1)$. Indeed, the gauge equivalence matrix 
should be simultaneously upper- and lower-triangular, since the conjugation by it should preserve 
the (opposite) triangular types of the main term matrices $K$ and $N$ at zero and at infinity, and each one of the latter matrices 
has distinct eigenvalues.  Therefore, the gauge transformation  is given by a diagonal matrix 
$\La=(\La_1,\La_2)$. 
It acts on the residue matrix by multiplying its off-diagonal terms $\pm a$ by ratios of eigenvalues: 
 by the numbers $\left(\frac{\La_1}{\La_2}\right)^{\pm1}$. Thus transformed residue matrix 
 has opposite off-diagonal elements, as does the initial matrix, if and only if the latter ratios are  
 equal to $\pm1$. That is, if and only if $\La=\diag(1,\pm1)$ up to scalar factor. The  
 gauge transformation in question does not change $s$, and it multiplies  both $\chi$, $a$ by the 
 same number $\pm1$. It does not change the monodromy-Stokes data, and hence, the cross-ratio $\mcr$. Conversely, let two systems (\ref{mchoyn}) with 
 polynomial solutions of degree $\ell\geq1$ have the same parameter value $s$, and let the corresponding 
 monodromy-Stokes data have the same cross-ratio. Then the corresponding monodromy 
 operators are the same. Indeed, one of the Stokes matrices at the origin is trivial: 
 this holds for every irregular non-resonant Poincar\'e rank 1 singularity, whenever one of the canonical solutions is holomorphic on $\cc^*$; in our case the 
 polynomial solution is such a solution. Therefore, in a canonical 
 sectorial solution basis at $0$ the monodromy is a triangular matrix with the same eigenvalues, as the formal monodromy, by (\ref{monprod}). The eigenvalues of the formal monodromy are  $e^{2\pi i\ell}=1$ and 
 $e^0=1$. Thus, the monodromy of system (\ref{mchoyn}) in question is unipotent. 
 Then it is uniquely determined by the projectivized monodromy, which in its turn 
 is uniquely determined by the cross-ratio $\mcr$, by Proposition \ref{abmon} and non-degeneracy of the monodromy--Stokes data, see Theorem  \ref{thpar}, Statement c). Finally, the 
 complete monodromy-Stokes data is uniquely determined by the cross-ratio $\mcr$. Thus, 
 the systems in question have the same formal normal forms at each singular point $p=0,\infty$ 
 and the same monodromy-Stokes data. Hence, they are gauge equivalent, by Theorem \ref{gaugeq}. 
 This together with the above characterization of gauge equivalent systems  (\ref{mchoyn}) 
 implies that the systems  in question are obtained one from the other by the 
 simultaneous sign change $(\chi,a)\mapsto(-\chi,-a)$. This proves Theorem \ref{threq}.
 \end{proof}

   \subsection{Realization of parabolic monodromy-Stokes data by vector bundles with connections} 
   
 Our goal is to prove Theorem  \ref{tsmooth}, which states that each determinantal surface $S_{\ell,\pm}$ is biholomorphically 
 parametrized (outside the locus $\{ s=0\}$) by the complement of $\cc^*\times\oc$ to an 
 analytic hypersurface (curve), and its fibers $\{ s=const\neq0\}$ are rational curves parametrized by 
 finitely punctured fibers $\{ s\}\times\oc$. The results of the previous subsection together with 
 \cite[appendix 3, lemma 1]{Bol18} imply that the parameters $(s,\mcr)$ yield a local biholomorphic 
 parametrization of each surface $S_{\ell,\pm}^o$. But the value of the parameter $\mcr$ depends on the 
 homotopy class of its defining pair of paths $\alpha_0$, $\alpha_\infty$ and the corresponding Stokes sectors 
 $S^0_0$, $S^\infty_0$; the paths and the sectors should be permuted by the involution $z\mapsto z^{-1}$. 
 As $s$ makes one turn around the origin, the Stokes sectors turn, the endpoints of the paths (which lie 
  in the Stokes sectors) also turn, and the homotopy class of each path changes. Therefore, as $s$ makes a turn while $\mcr$ remains constant, the monodromy-Stokes data defined by the initial paths 
  (and hence, the underlying system) changes in general. 
  
  To construct a global biholomorphic parametrization of the surfaces $S^o_{\ell,\pm}$, we pass to the 
  universal covering $\cc_t$ over the punctured line $\cc^*_s$, $s=e^t$, and construct  continuous families 
  of Stokes sectors $S_{0,t}^0$, $S_{0,t}^\infty$ and paths $\alpha_{0,t}$, $\alpha_{\infty,t}$. Then we construct 
  a holomorphic family $\mce_{t,\mcr}$ of holomorphic vector bundles with connections over $\oc_z$, parametrized by 
  $(t,\mcr)$, with formal normal forms  (\ref{fnf00}) and (\ref{fnfii}) at $0$ and at $\infty$ respectively, whose  monodromy-Stokes data 
  defined by the above paths and sectors is parabolic and has cross-ratio $\mcr$. We show that two bundles 
  $\mce_{t,\mcr}$ and $\mce_{t',\mcr'}$ with connections are gauge equivalent, if and only if 
  \begin{equation}(t',\mcr')=(t,\mcr)+(2\pi in, 2n), \  \ n\in\zz.\label{bundleq}
 \end{equation} 
 Afterwards, in the next subsection we pass to the quotient  of the space $\cc_t\times\oc_\mcr$ by the cyclic group 
 generated by the translation by the vector $(2\pi i,2)$ and get a family of bundles with connections 
 parametrized by $\cc^*_s\times\oc$. We show that all the parameters in $\cc^*\times\oc$ except for an 
 analytic hypersurface (curve) correspond to trivial bundles equipped with linear systems of type (\ref{mchoyn}) 
 with polynomial solutions. This will yield a biholomorphic parametrization of each surface $S_{\ell,\pm}^o$ 
 by the above parameters in $\cc^*_s\times\oc$. 

The above-mentioned families of sectors and paths are 
  $$S_{0,t}^0=\{\im t-\arg z\in(-\frac\pi2+\delta, \frac{3\pi}2-\delta)\}, \ \ 
 S_{0,t}^\infty=\{\im t+\arg z\in(-\frac\pi2+\delta, \frac{3\pi}2-\delta)\},$$
\begin{equation}\alpha_{0,t}(\tau)= e^{i\tau(Im t-\pi)}, \ \alpha_{\infty,t}(\tau)=\overline{\alpha_{0,t}(\tau)}, 
 \ \tau\in[0,1].\label{alft}\end{equation}
 They are permuted by the involution $z\mapsto z^{-1}$, and the sectors are Stokes for the formal normal 
 forms (\ref{fnf00}) and (\ref{fnfii}) respectively. This follows from definition.
 
  \begin{theorem} \label{tbundle} For every $\ell\in\nn$ there exists a two-dimensional 
  holomorphic vector bundle 
  $\mce=\mce_\ell$ over $\oc_z\times\cc_t\times\oc_\mcr$ such that for every 
  $t\in\cc$, set $s:=e^t$,  and every $\mcr\in\oc$ the restriction $\mce_{t,\chi}=\mce_{t,\chi,\ell}$ 
  of the bundle $\mce$ to  $\oc_z=\oc_z\times\{(t,\chi)\}$ is equipped with a meromorphic connection 
  depending holomorphically on $(t,\chi)$ that has  formal normal forms 
$$ Y'=\left(\frac{\diag(-\frac s2, 0)}{z^2}+\frac{\diag(\ell, 0)}z\right)Y \ \ \text{ at } 0,$$
 \begin{equation} 
 Y'=\left(\diag(0, \frac s2)+\frac{\diag(\ell, 0)}z\right)Y \ \ \text{ at } \infty\label{2fnf}\end{equation} 
and satisfies the following statements: 
 
a) The monodromy-Stokes data of the connection defined by the above Stokes sectors and paths is 
parabolic and has the given cross-ratio $\mcr$. 
 
 b) Two bundles $\mce_{t,\mcr}$, $\mce_{t',\mcr'}$ with connections corresponding to two distinct $(t,\mcr)$ and $(t',\mcr')$ are analytically equivalent, if and only if relation (\ref{bundleq}) holds.
 \end{theorem}
 \begin{proof} Fix a $\ell\in\nn$. Set  
 $$d_\ell=d_\ell(s):=\frac{\ell!}{2\pi i}\left(\frac2s\right)^\ell e^{\frac s2}.$$
 We consider the following model linear systems depending on parameters $s\in\cc^*$, $u\in\cc$: 
 \begin{equation}Y'=\left(\frac{\diag(-\frac s2,0)}{z^2}+\frac{\left(\begin{matrix}\ell & 0\\ ud_\ell & 0\end{matrix}\right)}z\right)Y \ \ \text{ over the disk } D_2:=\{|z|<2\};\label{sys0}\end{equation}
  \begin{equation}Y'=\left(\diag(0, \frac s2)+\frac{\left(\begin{matrix}\ell & ud_\ell \\ 0 & 0\end{matrix}\right)}z\right)Y \ \ \text{ over  } \oc\setminus\overline{D_{\frac12}}.\label{sysinf}\end{equation}
  These systems  have the prescribed formal normal forms (\ref{2fnf}) at the origin and at the infinity respectively. It is easy to check that their monodromy operators in appropriate local solution bases 
  are given by lower (upper) triangular unipotent matrices 
  with triangular elements being equal to $\pm u$. 
  We paste them together over the annulus 
  $$\mca:=\{\frac12<|z|<2\}=D_2\cap(\oc\setminus\overline{D_{\frac12}})$$ 
  in appropriate way depending on $(t,\mcr)$, preserving the 
  differential equations as connections on trivial bundles, in order to obtain 
  a holomorphic vector bundle $\mce_{t,\mcr}$ with connection having the desired cross-ratio $\mcr$ of the 
  monodromy-Stokes data. Recall that $\mcr$ is the cross-ratio between appropriate analytic 
  extensions of canonical sectorial solutions. To define the gluing, we first find the canonical solutions and their analytic extensions to  $z_0=1$ along the paths $\alpha_{0,t}^{-1}$, $\alpha_{\infty,t}^{-1}$ and find 
  monodromy matrices of systems (\ref{sys0}), (\ref{sysinf}). To this end, let us 
  introduce the paths 
  $$\beta_{0,t}(\tau)=(1-\tau)\alpha_{0,t}(1)=(1-\tau)e^{i(\im t-\pi)},$$
  $$\beta_{\infty,t}(\tau)=\frac1{\beta_{0,t}(\tau)}=\frac1{1-\tau}e^{-i(\im t-\pi)},$$
  $$\gamma_{p,t}:=\beta_{p,t}^{-1}\circ\alpha_{p,t}^{-1}$$
  Each $\gamma_{p,t}$ is a path  from the point $p$ to 1. 
  
  We will deal with functions defined on a disk $U\subset\cc^*$ 
  centered at $z_0=1$. For every $z\in U$ by $\gamma_{p,t,z}$ we will denote the composition of the path  $\gamma_{p,t}$ from  $p$ to 1 with the straight path from 1 to $z$. 
  
  The canonical sectorial basic solutions of system (\ref{sys0}) at $0$ in $S_{0,t}^0$  numerated by 
  the main term matrix eigenvalues $-\frac s2$ and $0$ will be denoted by $f_{10}$, $f_{20}$. 
  Similarly the canonical basic solutions of system (\ref{sysinf}) at $\infty$ in $S_{0,t}^{\infty}$ 
  numerated by 
  the main term matrix eigenvalues $0$ and $\frac s2$ will be denoted by $f_{1\infty}$, $f_{2\infty}$.
  
  \begin{proposition} \label{cansol} 1) The analytic extensions to the disk $U$ 
  of sectorial solutions $f_{10}$, $f_{20}$ 
  of system (\ref{sys0}) and $f_{1\infty}$, $f_{2\infty}$ for system (\ref{sysinf}) along paths 
  $\alpha_{0,t}$ and $\alpha_{\infty,t}$ respectively (after their appropriate normalizations by 
  constant factors)  are given  by the following formulas:
  \begin{equation} f_{10}(z)=(e^{\frac s2(\frac1z-1)}z^\ell, \ d_\ell u \int_{\gamma_{0,t,z}}\zeta^{\ell-1}e^{\frac s{2}(\frac1\zeta-1)}d\zeta), \ \ \ f_{20}(z)=(0,1),\label{ff0}\end{equation}
  \begin{equation} f_{1\infty}(z)=(z^\ell,0), \ \ \ f_{2\infty}(z)=(d_\ell u z^{\ell}\int_{\gamma_{\infty,t,z}}
  \zeta^{-(\ell+1)}e^{\frac s2(\zeta-1)}d\zeta, \ e^{\frac s2(z-1)}).\label{ffinf}\end{equation}
  
  2) In the above canonical bases the monodromy matrices $M_0$, $M_\infty$ of systems (\ref{sys0}) and (\ref{sysinf}) respectively are equal to 
  \begin{equation} M_0=\left(\begin{matrix} 1 & 0 \\ u & 1\end{matrix}\right), 
   \ M_\infty= \left(\begin{matrix} 1 & u \\ 0 & 1\end{matrix}\right).\label{mndr}\end{equation}
  \end{proposition}
  \begin{proof} The vector functions $f_{20}$ and $f_{1\infty}$ are clearly solutions 
  of systems (\ref{sys0}) and (\ref{sysinf}) respectively. They are obviously canonical, since they are respectively 
   holomorphic at zero and meromorphic at infinity. Let us show that the other vector functions 
   $f_{10}$, $f_{2\infty}$ are canonical. First let us prove their well-definedness: convergence of the 
   corresponding integrals. The integral in (\ref{ff0}) has subintegral expression 
   $\zeta^{\ell-1}e^{\frac s2(\frac1\zeta-1)}$, which has singular point 
   at the origin. It is taken along the path $\gamma_{0,t,z}$, which goes 
   out from the origin in the straightline direction with azimuth $\im t-\pi$ with respect to the 
   real axis (as does the path $\beta_{0,t}^{-1}$). Therefore,  along the latter straightline part of 
   the integration path one has
   $$\arg\zeta=\im t-\pi, \ \  \frac s{2\zeta}=\frac{e^t}{2|\zeta|e^{i(\im t-\pi)}}=-\frac{e^{\re t}}{2|\zeta|}.$$
  Therefore, the subintegral exponent is a flat function along the path $\gamma_{0,t,z}$ at the origin, 
  and hence, the integral in (\ref{ff0}) converges. Convergence of the integral in (\ref{ffinf}) is proved 
  analogously. The fact that the vector functions $f_{10}$ and $f_{2\infty}$ are solutions of systems 
  (\ref{sys0}) and (\ref{sysinf}) respectively is proved by straightforward differentiation. Let us prove 
  that their  analytic extensions along the paths $\alpha_{0,t}$ and 
  $\alpha_{\infty,t}$ are the canonical sectorial solutions in $S_{0,t}^0$ and $S_{0,t}^{\infty}$ 
  respectively that are numerated by the main term matrix eigenvalues $-\frac s2$ and $\frac s2$. 
  Let us prove this statement for $f_{10}$: for the vector function $f_{2\infty}$ the proof is 
  analogous.  It suffices to prove the latter statement for the backward analytic extension 
   along $\gamma_{0,t,z}^{-1}$ of the vector function $f_{10}$  towards a point $w\in\gamma_{0,t,z}$ close to the origin with $\arg w=\im t-\pi$. The backward extension is 
  given by the integral (\ref{ff0}) taken along the 
  straightline segment $[0,w]$. The integral is flat in $w$, as is the subintegral function, see the 
  above discussion, and the first component in $f_{10}$ is also flat for the same reason. 
  Therefore,  $f_{10}(w)=o(1)$, as $w\to0$ with $\arg w=\im t-\pi$, while $f_{20}\equiv(0,1)$.   
  The canonical solution in $S_{0,t}^0$ is uniquely determined by the condition that it should 
 tend to zero, as $w\to0$ along a ray along  which $e^{\frac s{2w}}$  is flat, since the other 
 canonical basic solution $(0,1)$ does not tend to zero. Hence, $f_{10}$ is a canonical solution. 
 
   Now for the proof of the proposition it remains to calculate the monodromy. Clearly the monodromy 
   acts trivially on the functions $f_{20}$ and $f_{1\infty}$. Let us calculate the monodromy images 
   of the other functions $f_{10}$ and $f_{2\infty}$. As $z$ makes one courterclockwise circuit 
   around the origin, the integral in (\ref{ff0}) changes by the additive constant 
   $$2\pi ie^{-\frac s2} Res_0(z^{\ell-1}e^{\frac s{2z}})=\frac{2\pi i}{\ell!}e^{-\frac s2}(\frac s2)^\ell=
   d_\ell^{-1}.$$
   This together with (\ref{ff0}) implies the formula for the monodromy matrix $M_0$ 
   in (\ref{mndr}). Similarly, as $z$ makes a counterclockwise circuit, the integral in (\ref{ffinf}) 
   is changed by the additive constant equal to $2\pi ie^{-\frac s2}$ times the Laurent coefficient 
   at $\zeta^{-1}$ of the function $\zeta^{-(\ell+1)}e^{\frac s{2}\zeta}$. The latter 
   Laurent coefficient is equal to $\frac{(\frac s2)^\ell}{\ell!}$. Thus, the resulting additive constant is 
   equal to $d_\ell^{-1}$. This implies the formula for the monodromy matrix $M_\infty$ in 
   (\ref{mndr}). Proposition \ref{cansol} is proved.
   \end{proof}
  
  Let 
  $$\mcy^0(z)=(f_{10},f_{20})(z), \ \mcy^\infty(z)=(f_{1\infty},f_{2\infty})(z)$$
  denote the fundamental matrices formed by the canonical solution bases (\ref{ff0}) and 
  (\ref{ffinf}) of systems (\ref{sys0}) and (\ref{sysinf}) respectively. Here the vector functions 
  $f_{jp}(z)$ are treated as columns of the corresponding matrix $\mcy^p$. Their 
 monodromy matrices $M_0$ and $M_\infty$ given by (\ref{mndr}) are conjugated: 
  \begin{equation} M_0=\Psi_\zeta^{-1} M_\infty\Psi_\zeta, \ \Psi_\zeta:=
  \left(\begin{matrix} 0 & 1 \\ 1 & 0\end{matrix}\right)\left(\begin{matrix}1 & 0\\ -\zeta & 1\end{matrix}
  \right).\label{psiz}\end{equation}
  Here $\zeta$ is an arbitrary complex number. Therefore, the rescaled fundamental matrix 
  \begin{equation}\wt\mcy^\infty(z)=\mcy^\infty(z)\Psi_\zeta\label{defwt}\end{equation}
  of system (\ref{sysinf}) has the same monodromy matrix $M_0$, as the fundamental 
  matrix $\mcy^0(z)$. 
  
  Now we paste  the fundamental matrices $\mcy^0(z)$ and 
  $\wt\mcy^\infty(z)$ together over the annulus $\mca:=D_2\setminus\overline{D_{\frac12}}$ 
  in order to paste  the  systems 
  (\ref{sys0}) and (\ref{sysinf}) to get  a holomorphic vector bundle with connection induces 
  by the pasted systems.  Namely, the direct products 
  $$\mcf_0:=\cc^2\times D_2, \ \ \mcf_\infty:=\cc^2\times 
  (\oc\setminus\overline{D_{\frac12}})$$ 
  are identified as follows: for every $z\in\mca$,  
  $$(v,z)\in\mcf_0 \text{ is equivalent to } (w,z)\in\mcf_\infty, \ \text{ if and only if}$$ 
  \begin{equation} 
  v=X_{t,u,\zeta}(z)w, \ X_{t,u,\zeta}(z):=\mcy^0(z)\left(\wt\mcy^\infty(z)\right)^{-1}.\label{bundlegl}
  \end{equation}
  Here the  matrix function $X_{t,u,\zeta}(z)$ is defined by  (\ref{bundlegl}) for $z$ 
  close to 1, and for the other $z$ it is defined by its analytic extension along a path going 
  from 1 to $z$ in the annulus $\mca$. 
  
  \begin{proposition} \label{pglbund} 1) The matrix function $X_{t,u,\zeta}(z)$ is holomorphic on the annulus $\mca$,  
  and it depends holomorphically on the parameters $(t,u,\zeta)\in\cc^3$.
  
  2) The corresponding gluing (\ref{bundlegl}) pastes the direct products $\mcf_0\times\cc^3_{t,u,\zeta}$ and $\mcf_\infty\times\cc^3_{t,u,\zeta}$ 
  to a holomorphic vector bundle over $\oc_z\times\cc^3_{t,u,\zeta}$, which will be denoted by 
  $\wt\mce$ and treated as a holomorphic family of holomorphic vector bundles $\mce_{t,u,\zeta}$ on $\oc_z$
  
  3) Gluing (\ref{bundlegl}) sends the connection on $\mcf_\infty$ defined by system (\ref{sysinf}) to 
  the connection on $\mcf_0$ defined by (\ref{sys0}), and hence, induces a meromorphic connection 
  $\nabla_{t,u,\zeta}$ on each bundle $\mce_{t,u,\zeta}$; the family of connections 
   $\nabla_{t,u,\zeta}$ depends holomorphically on the parameters $(t,u,\zeta)$.
  
  4) The monodromy-Stokes data of each connection  $\nabla_{t,u,\zeta}$ defined by the paths 
  $\alpha_{0,t}$ and $\alpha_{\infty,t}$  is parabolic. It is non-degenerate, if and only if $(u,\zeta)\neq(0,0)$, and in this case the corresponding cross-ratio is equal to 
  \begin{equation}\mcr=\frac\zeta{u}.\label{mcrform}\end{equation}
  \end{proposition}
  \begin{proof} Holomorphicity of the matrix function $X_{t,u,\zeta}(z)$ on the annulus $\mca$, 
  i.e., its single-valuedness (independence on path of analytic extension) follows from 
  coincidence of monodromy matrices of the fundamental matrices $\mcy^0$ and $\wt\mcy^\infty$: 
  as $z$ makes a counterclockwise circuit around the origin, both fundamental matrices 
  are multiplied by $M_0$, and in the ratio $X_{t,u,\zeta}(z)=\mcy^0(\wt\mcy^\infty)^{-1}$, this $M_0$ cancels out. This proves Statement 1), which in its turn (together with holomorphicity of the 
  fundamental matrices in the parameters $t$, $u$, $\zeta$) implies Statement 2). Gluing 
  (\ref{bundlegl}) sends each column of the fundamental matrix solution $\wt\mcy^\infty$ of system 
  (\ref{sysinf}) to the corresponding column of the fundamental matrix solution $\mcy^0$ of system 
  (\ref{sys0}), by definition. Hence, it pastes the systems to one meromorphic connection in 
  the bundle $\mce_{t,u,\zeta}$, which depends holomorphically on the parameters, as do the 
  systems and the fundamental matrices. The second columns of both fundamental matrices 
  are canonical solutions $f_{1\infty}$ and $f_{20}$, and they are pasted together. This implies 
  parabolicity of the monodromy-Stokes data of the pasted connection. Its non-degeneracy criterium  
  $(u,\zeta)\neq(0,0)$ follows from construction.
  
  Let us now  calculate the cross-ratio $\mcr$ of the monodromy-Stokes data in the case, when 
  $(u,\zeta)\neq(0,0)$. To do this, let us find the values $f_{jp}(1)$ of the 
  canonical solutions after identification (\ref{bundlegl}) in the trivial bundle chart $\mcf_0$ and 
  their projectivizations $q_{jp}$ and the point $q_{10}'$. Note that 
  \begin{equation}\wt\mcy^\infty=(f_{2\infty}-\zeta f_{1\infty}, \ f_{1\infty}):\label{mcyi}
  \end{equation}
  the columns of the fundamental matrix $\wt\mcy^\infty$ are the vector functions $f_{2\infty}-\zeta f_{1\infty}$ and $f_{1\infty}$, by (\ref{psiz}) and (\ref{defwt}).  And they are identified with 
  $f_{10}$ and $f_{20}$. Therefore, $f_{1\infty}(1)$ is identified with $f_{20}(1)$, and 
  $f_{2\infty}(1)$ is identified with $f_{10}(1)+\zeta f_{20}(1)$.  
  Let $\wt f_{10}(z)$ denote the result of counterclockwise analytic    extension 
  of the vector function $f_{10}(z)$ considered as a germ of vector function at $z=1$. One has 
   $$\wt f_{10}(1)=f_{10}(1)+uf_{20}(1),$$
   by formula (\ref{mndr}) for  the monodromy matrix $M_0$. Thus, in the chart 
   $\mcf_0$, the four points $q_{20}$, $q_{10}$, $q_{10}'$, $q_{2\infty}$ are the projections to 
   $\cp^1=\oc$ of the four vectors $f_{20}(1)$, $f_{10}(1)$, $\wt f_{10}(1)=f_{10}(1)+uf_{20}(1)$, 
   $f_{10}(1)+\zeta f_{20}(1)\simeq f_{2\infty}(1)$. In the basis $f_{10}(1)$, $f_{20}(1)$ 
   (let $(w_1,w_2)$ denote the corresponding coordinates on $\cc^2$)  the 
   latter vectors are $(0,1)$, $(1,0)$, $(1,u)$, $(1,\zeta)$. Their projections  to 
   $\cp^1=\oc_\Phi$, $\Phi=\frac{w_2}{w_1}$, are the points $\infty$, $0$, $u$, $\zeta$. Their 
   cross-ratio (\ref{crat}) is equal to $\frac{\zeta}u$. This proves (\ref{mcrform}) and finishes the proof of 
   Proposition \ref{pglbund}. 
     \end{proof}
 \begin{proposition} \label{proeq} For every given $t\in\cc$  two  holomorphic vector bundles 
 with connections $(\mce_{t,u,\zeta},\nabla_{t,u,\zeta})$, $(\mce_{t,u',\zeta'},\nabla_{t,u',\zeta'})$ corresponding to the same $t$ and some $(u,\zeta),(u',\zeta')\neq(0,0)$ are holomorphically gauge equivalent, if and only if the corresponding cross-ratio are the same, i.e., if and only if 
 $[u:\zeta]=[u':\zeta']$ as points of the projective line $\cp^1$. 
 \end{proposition}
 \begin{proof} The bundles in question have the same formal normal forms at each singular point. 
The paths $\alpha_{tp}$, $p=0,\infty$, defining their monodromy-Stokes data 
are the same for both systems, since they depend only on $t$. Therefore, the 
bundles are gauge equivalent, if and only if their monodromy-Stokes data are the same. 
Thus, gauge equivalence  implies the cross-ratio equality. Let us prove the converse. 
Recall that the monodromy--Stokes data is parabolic and non-degenerate, and the cross-ratio determines it uniquely up to multiplication of the monodromy operator by constant factor (Proposition \ref{abmon}). In both connections the monodromy operators are unipotent. 
Therefore, equality of cross-ratios implies coincidence of the monodromy-Stokes data, and hence, 
gauge equivalence. Proposition \ref{proeq} is proved.
\end{proof}

In what follows the gauge equivalent bundles $\mce_{t,u,\zeta}$ with the given value 
$\mcr:=\frac\zeta{u}$ will be denoted by one and the same symbol
$$\mce_{t,\mcr}:=\mce_{t,u,\zeta}, \ \mcr=\frac\zeta{u}.$$
\begin{proposition} \label{bundlee} The family of vector bundles $\mce_{t,\mcr}$ form a holomorphic 
vector bundle over $\oc_z\times\cc_t\times\oc_\mcr$, which will be denoted by $\mce$, with 
a holomorphic family of meromorphic connections over $\oc_z$ with irregular nonresonant 
singular points at zero and at infinity.
\end{proposition} 
\begin{proof} The vector bundle family $\mce_{t,u,\zeta}$  form a holomorphic 
vector bundle (denoted $\wt\mce$) over $\oc\times\cc_t\times(\cc^2_{u,\zeta}\setminus\{(0,0)\})$ with a holomorphic family of connections over $\oc$. Namely, 
the bundle $\wt\mce$ is defined by gluing (\ref{bundlegl}) of two trivial bundles $\cc^2\times D_2
\times\cc_t\times(\cc^2\setminus\{(0,0)\})$ and $\cc^2\times (\oc\setminus\overline{D_{\frac12}})
\times \cc_t\times(\cc^2\setminus\{(0,0)\})$: the gluing matrix function $X_{t,u\zeta}(z)$ 
is holomorphic in $(z,t,u,\zeta)$. Consider the tautological bundle $\pi:\cc^2_{u,\zeta}\setminus\{(0,0)\}\to\cp^1=\oc_\mcr$, $\mcr=\frac\zeta{u}$, and its two sections $\sigma_0$, $\sigma_\infty$ 
over the two affine charts $\cc=\{\mcr\neq\infty\}$ and $\{\mcr\neq0\}$: 
$$\sigma_0:\mcr\mapsto(1,\mcr), \ \  \sigma_1: \mcr\mapsto(\mcr^{-1},1).$$
Let $\mce^0$ and $\mce^\infty$ denote the pullbacks of the bundle $\wt\mce$ under the 
maps $(z,t,\mcr)\mapsto(z,t,\sigma_p(\mcr))$, $p=0,\infty$. This yields two  holomorphic 
vector bundle families $\mce^p_{t,\mcr}$ over $\oc_z$, with $p=0,\infty$, $\mcr\neq0,\infty$. 
 For every $(t,\mcr)\in\cc^2$ with $\mcr\neq0$ the bundles $\mce^0_{t,\mcr}$ and 
 $\mce^\infty_{t,\mcr}$ are isomorphic as holomorphic bundles with connections. The isomorphism 
 is  uniquely determined by normalization so that in the trivializing chart $\mcf^0$ it sends the 
 solution $f_{20}(z)=(0,1)$ of the connection in one bundle to the same solution of the connection in 
 the other bundle. Uniqueness follows from non-degeneracy of the monodromy-Stokes data. Indeed, the  restriction to the fiber $\oc\times\{ z=1\}$ of the projectivized isomorphism should 
 map the set $\{ q_{20}, q_{10}, q_{10}', q_{2\infty}\}$ corresponding to one connection to 
 the similar set corresponding to the other connection. The collection of points in each set is at least 
 three, by non-degeneracy. This defines the projectivized isomorphism uniquely. Thus, 
 the bundle isomorphism is uniquely defined up to constant factor, 
 and the latter factor is uniquely defined by the above normalization. The isomorphism thus constructed depends 
 holomorphically on the parameters, as does the monodromy-Stokes data. Therefore, 
 the isomorphism family thus constructed pastes the bundles 
 $\mce^0$ and $\mce^\infty$ together to one holomorphic vector bundle $\mce$. Proposition 
 \ref{bundlee} is proved.
 \end{proof}

\begin{lemma} \label{probeq} Two bundles $\mce_{t,\mcr}$ and $\mce_{t',\mcr'}$ with connections are holomorphically gauge equivalent, if and only if relation (\ref{bundleq}) holds.
\end{lemma}
\begin{proof} Clearly, the bundles with connections are gauge equivalent, if and only if they have the same formal normal forms and the same monodromy-Stokes data (or equivalently, the same cross-ratio 
$\mcr$) defined by one and the same pairs of paths and Stokes sectors. 
Coincidence of formal normal forms is equivalent to the equality 
saying that $s=e^t$ is equal to $s'=e^{t'}$. This holds, if and only if 
$t'-t=2\pi in$, $n\in\zz$. The above monodromy-Stokes data of bundles $\mce_{t,\mcr}$ and 
$\mce_{t',\mcr'}$ are defined by two different pairs of paths: $(\alpha_{0,t},\alpha_{\infty,t})$ and 
$(\alpha_{0,t'},\alpha_{\infty,t'})$ Let us show that the monodromy-Stokes data  defined by the first pair of paths 
 is the same for both 
connections, if and only if $\mcr'-\mcr=2n$. To do this, we use the 
 next proposition, which shows how does the cross-ratio change when we replace its defining paths 
 $\alpha_{p,t}$ by $\alpha_{p,t'}$, $p=0,\infty$.  

\begin{proposition} \label{crossm} Let $\mce$ be a holomorphic vector bundle with connection on 
$\oc$ having formal normal 
forms (\ref{2fnf}) at the origin and at infinity that is not ``diagonal'', i.e.,  not a direct sum of one-dimensional holomorphic bundles with connections. Let the connection have a meromorphic solution. 
 Let $s$ be the corresponding parameter 
from the formal normal forms. For every value $t=\ln s$ of the logarithm let $\alpha_{p,t}$, $p=0,\infty$, be the paths defined in (\ref{alft}), and let $\mcr_t$ denote the characteristic cross-ratio 
of the corresponding monodromy-Stokes data. One has 
\begin{equation}\mcr_{t+2\pi i}=\mcr_t+2.\label{mcrt}\end{equation}
\end{proposition}
\begin{proof} When we add $2\pi i$ to $t$, the new path $\alpha_{0,t}$ is obtained from the old one 
by multiplication by counterclockwise circuit. This means that  the  analytic continuation along the 
new path 
$\alpha_{0,t}^{-1}$ of the canonical sectorial solution $f_{10}$ (flat section with respect to the 
connection) is obtained from that for the old 
path by applying the inverse of the monodromy. Similar statement holds for 
$\alpha_{\infty,t}$ and $f_{2\infty}$: the new $f_{2\infty}$ is obtained from the old one 
by applying the monodromy. Recall that the monodromy-Stokes data defined by the paths 
$\alpha_{0,t}$ and $\alpha_{\infty,t}$ are parabolic and non-degenerate, by Proposition \ref{pbng}. We consider that $f_{2\infty}$ is not 
proportional to  $f_{10}$; the opposite case is treated analogously (and can be also deduced by passing to limit in a family of pasted bundles with connections).  
Thus, $q_{2\infty}\neq q_{10}$, and $q_{10}$, $q_{20}$, $q_{2\infty}$ are distinct. 
Choosing appropriate coordinate on $\oc$ we 
consider that the monodromy-Stokes data for the old $t$ are given by 
$$q_{10}=0, \ q_{20}=q_{1\infty}=\infty, \ q_{10}'=u, \ q_{2\infty}=1.$$ 
Then for the new $t$ one has 
$$q_{10}=-u,  \ q_{20}=q_{1\infty}=\infty,  \ q_{10}'=0, \ q_{2\infty}=1+u,$$
by the above path and monodromy argument. Therefore, the old cross-ratio (\ref{crat}) is equal to 
$\frac1u$, and the new one is equal to $\frac1u+2$.
This proves the proposition.
\end{proof}

Proposition \ref{crossm} together with the above argument implies Lemma \ref{probeq}.
\end{proof}

Propositions \ref{pglbund}, \ref{proeq}, \ref{bundlee}  and Lemma \ref{probeq}  together imply Theorem \ref{tbundle}.
  \end{proof}
  
  \subsection{Quotient bundle family. Global parametrization of surfaces $S_{\ell,\pm}^o$}
  Consider the quotient of the space $\cc_t\times\oc_\mcr$ by the action of the cyclic group 
   generated by the translation $F$ by the vector $(2\pi i, 2)$. 
  The quotient space $(\cc\times\oc)\slash F$ is identified with 
  $\cc^*_s\times\oc_\xi$, and the quotient projection is 
 $$\pi_{tr}:(t,\mcr)\mapsto(s=e^{t}, \ \xi=\pi i\mcr-t).$$

  \begin{lemma} \label{tbu} There exists a holomorphic vector bundle $\wh\mce$ over $\oc_z\times\cc^*_s\times\oc_\xi$ with a family of meromorphic connections over $\oc_z$ such that 
  the bundle $\mce=(\mce_{t,\mcr})$ is its pullback  under the propection $z\times(t,\mcr)\mapsto z\times\pi_{tr}(t,\mcr)$.
  \end{lemma}
  \begin{proof} For every $(t,\mcr)$ and $(t',\mcr')=F^n(t,\mcr)$, 
   $n\in\zz$, there is a unique isomorphism of the corresponding bundles 
  $\mce_{t,\mcr}$ and $\mce_{t'_n,\mcr'_n}$ with connections up to scalar factor. This isomorphism can be 
  normalized as in the proof of Proposition \ref{bundlee}, and then it depends locally 
  holomorphically on the parameters $(t,\mcr)$, as in the same proof. This  implies the lemma analogously to the same proof. 
  \end{proof}
 
 Consider the bundle $\wh\mce$ as a family of holomorphic vector bundles $\wh\mce_{s,\xi}$ over 
 $\oc_z$ with connections. Let $\ell\in\nn$ be the corresponding parameter in the formal normal forms  
 of the connections. Set 
 $$\triv_\ell:=\{(s,\xi)\in\cc^*\times\oc \ | \ \text{ the 
 holomorphic  bundle } \ \wh\mce_{s,\xi} \ \text{  is trivial}\},$$ 
 $$\Sigma_\ell:=(\cc^*\times\oc)\setminus\triv_\ell.$$
 Let $\Sigma_\ell'\subset\triv_\ell$ denote the subset of those 
 $(s,\xi)$ for which in the corresponding 
 connection in the trivial bundle on $\oc$ the line non-coincidence condition does not hold: that is, 
  the eigenline of the main term matrix at $0$ with the eigenvalue $-\frac s2$ 
 coincides with the eigenline of the main term matrix at infinity with the eigenvalue $\frac s2$.  
 \begin{theorem} \label{tbtriv} 1) The subset $\triv_\ell\subset\cc^*\times\oc$ is the complement 
 to a complex analytic hypersurface (curve) $\Sigma_\ell$: a complex analytic subset of pure codimension 1. 
 
 2) The subset $\Sigma'_\ell\subset\triv_\ell$ is also a complex analytic hypersurface (curve), in particular, 
 the subsets $\triv_\ell$ and  
$$\triv_\ell^o:=\triv_\ell\setminus\Sigma'_\ell\subset\cc^*_s\times\oc_\xi$$
 are both open connected and dense in $\cc^*_s\times\oc_\xi$. 
 
3) For every $(s,\xi)\in\triv_\ell^o$  the  trivial bundle $\wh\mce_{s,\xi}$ with connection is analytically gauge equivalent to a system 
 of type (\ref{mchoyn}) having a polynomial solution. The parameters $(\chi,a)$ 
 of the corresponding 
 system (\ref{mchoyn}) are uniquely determined up to sign change $(\chi, a)\mapsto(-\chi,-a)$ 
 and $(\chi,a)\neq(0,0)$.

 4) Thus defined two-valued  map $\mct:\triv_\ell^o\to\cc^3_{\chi,a,s}$ associating to each 
 $(s,\xi)$ the parameters of the corresponding system (\ref{mchoyn}) from Statement 3), 
has two holomorphic branches $\mct_{\pm}$  bijectively parametrizing the two determinantal 
 surfaces $S_{\ell,\pm}^o=\{\det(G_{1,\ell}(\chi,a,s)\pm G_{2,\ell}(\chi,a,s))=0\}\setminus\{ s=0\}$ by the domain $\triv_\ell^o\subset\cc^*\times\oc$. 
  \end{theorem}
  
  \begin{proof} The subset $\triv_\ell^o\subset\cc^*\times\oc$ is non-empty. Indeed, the set 
  $S_{\ell,+}\cup S_{\ell,-}$ of  
systems (\ref{mchoyn}) with polynomial solutions is non-empty, since the spectral curve 
$\Gamma_\ell^o$ is non-empty. Their monodromy-Stokes data 
corresponding to symmetric sectors and paths are parabolic,  and the corresponding parameters 
$(s,\xi)$ clearly lie in $\triv_\ell^o$.  This implies that $\Sigma_\ell$ is an analytic hypersurface, by 
 \cite[proposition 4.1]{malgr} (see also its short proof in 
  \cite[appendix 3]{Bol18}); see also a more general  \cite[theorem 2.2, p.449]{rohrl}. 
  The set $\Sigma_\ell'$ is the set of those parameters $(s,\xi)\in\triv_\ell$ that 
  correspond to those connections in trivial bundle where appropriate main term matrix 
  eigenlines coincides. This is a non-trivial codimension one condition, since the complement 
  $\triv_\ell^o=\triv_\ell\setminus\Sigma_\ell'$ has a non-empty interior. Therefore, $\Sigma_\ell'$ is 
  a complex hypersurface in $\triv_\ell$. Statements 1) and  2) are proved.
  
  \begin{proposition} \label{disjax} 
 The surfaces $S_{\ell,\pm}^o$ are disjoint from the $s$-axis $\{\chi=a=0\}$.
   \end{proposition}
   \begin{proof} Way 1. System (\ref{mchoyn}) with $\chi=a=0$ is diagonal with main term matrices 
   $\diag(-\frac s2,0)$ and $\diag(0,\frac s2)$ at zero and at infinity respectively. Therefore, 
   if $s\neq0$, then it has no vector polynomial solution, being the direct sum of one-dimensional equations with solutions $z^\ell e^{\frac s{2z}}$ and $e^{\frac s2z}$. 
   
   Way 2. For $\chi=a=0$ one has $G_{2,\ell}=0$, and $G_{1,\ell}=G_{1,\ell}\pm G_{2,\ell}$ is upper-triangular with non-zero diagonal elements $\frac s2$ and $\frac 12$. Hence, its determinant is non-zero.
 \end{proof}

  Statement 3) follows from Theorems \ref{thpar} and \ref{threq}; inequality $(\chi,a)\neq(0,0)$ follows 
  from Proposition \ref{disjax}. 
  
  Let us prove Statement 4). First, each point $(s_0,\xi_0)\in\triv_\ell^o$ has a neighborhood where the 
  map $\mct:(s,\xi)\mapsto(\chi,a,s)$ has two local holomorphic 
  branches.  This follows from \cite[appendix 3, theorem 2]{Bol18}, which implies holomorphicity of 
  normalized trivialization of the bundle $\mce_{s,\xi}$ in the parameter $(s,\xi)$ and hence, 
  holomorphicity in $(s,\xi)$ of the family of linear systems thus obtained (images of the connections) 
  in the trivialized bundle. The above two holomorphic branches clearly differ by sign 
(Statement 3)), and hence, their images lie in different surfaces $S_{\ell,\pm}$, since 
the latter surfaces are permuted by the sign change $(\chi,a)\mapsto(-\chi,-a)$. 
The parameter subset $\triv_\ell^o$ is open and connected, being the 
complement of the manifold $\cc^*\times\oc$ to the 
  union $\Sigma_\ell\cup\Sigma_\ell'$. Therefore,  analytic extension of each initial holomorphic branch of the map $\mct$ along every path in $\triv_\ell^o$ may lie  only in one 
  of the surfaces $S_{\ell,\pm}$. This together with the above statement implies that each branch yields a global holomorphic map 
  \begin{equation}\mct_{\pm}:\triv_\ell^o\to S_{\ell,\pm}^o=S_{\ell,\pm}\setminus\{ s=0\}.\label{mct}\end{equation}
  Its bijectivity follows from Statement 3) and the fact that each system (\ref{mchoyn}) with 
  $s\neq0$ having a polynomial solution has non-degenerate 
  parabolic monodromy-Stokes data and hence, 
  corresponds to a connection defined by some $(s,\xi)\in\triv_\ell^o$ (Theorem \ref{thpar}). Theorem \ref{tbtriv} is proved.
  \end{proof} 
\subsection{Regularity and rationality.  Proof of Theorem \ref{tsmooth}} 
  \begin{lemma} \label{lreg} The surfaces $S_{\ell,\pm}^o$ are disjoint regular submanifolds in 
  $\cc^3\setminus\{ s=0\}$ diffeomorphically parametrized by  the bijective maps $\mct_{\pm}$ given by (\ref{mct}). In particular, they are connected  and hence, irreducible.
  \end{lemma}
  \begin{proof} The maps $\mct_{\pm}$ are holomorphic and bijective. The surfaces $S_{\ell,\pm}^o$ are analytic subsets in $\cc^3\setminus\{ s=0\}$. To prove that they are  submanifolds, 
  it suffices to show that 
  the maps $\mct_{\pm}$ are immersions, and it suffices to prove this for one sign, say "$+$". (Their connectivity 
  (hence, irreducibility) follows from connectivity of the parametrizing domain $\triv_\ell^o$.) 
  To this end, we fix some $(s_0,\xi_0)\in\triv_\ell^o$ and consider 
  the new local coordinates $(s,\mcr)$ on its neighborhood, where $\mcr=-i\pi^{-1}(\xi+t)$ 
 is the characteristic cross-ratio of the monodromy-Stokes data corresponding to the paths 
  $\alpha_{0,t}$, $\alpha_{\infty,t}$, $t=\ln s$. Set $(\chi_0,a_0,s_0)=\mct_+(s_0,\xi_0)\in S_{\ell,+}^o$. 
   We claim that the value of the cross-ratio $\mcr$ is well-defined on all systems (\ref{mchoyn}) 
   with parameters $(\chi,a,s)$ close enough to $(\chi_0,a_0,s_0)$, and it is a holomorphic 
   function of $(\chi,a,s)$. Indeed, recall that $\mcr$ is the cross-ratio of the points 
   $q_{20}$, $q_{10}$, $q_{10}'$, $q_{2\infty}$, where $q_{10}$, $q_{20}$, $q_{2\infty}$ are projectivizations 
   of the values at $z=1$ of the canonical solutions $f_{10}$, $f_{20}$, $f_{2\infty}$, and $q_{10}'$ 
   is the projectivization of the value at 1 of the image  of the solution $f_{10}$ under the monodromy. 
   All the above four points $q$ depend holomorphically on the parameters of the system, as do 
   the canonical solutions. At least three of them are distinct for $(\chi,a,s)=(\chi_0,a_0,s_0)$ (non-degeneracy) and hence, for all close $(\chi,a,s)$. Therefore, $\mcr$ is 
   well-defined and holomorphic. This yields an inverse map 
   $(\chi,a,s)\mapsto(s,\mcr)$ holomorphic on a neighborhood of the point $(\chi_0,a_0,s_0)$ in  $\cc^3$. It is inverse to $\mct_+$: its pre-composition with $\mct_+$ is the identity. This implies 
   that $\mct_+$ is an immersion. 
   
   Let us show that the surfaces 
   $S_{\ell,\pm}^o$ are disjoint. Suppose the contrary: they are intersected at some 
   point $(\chi_0,a_0,s_0)$, $s_0\neq0$. One has  $(\chi_0,a_0)\neq(0,0)$, by Proposition 
   \ref{disjax}, and 
   $$(\chi_0,a_0,s_0)=\mct_{\pm}(s_0,\xi_{\pm}) \ \text{ for some } \ \xi_{\pm}\in\oc.$$
   But $\xi_+=\xi_-$, since they are equal (up to known multiplicative and additive constants) 
    to the characteristic cross-ratio of parabolic 
   monodromy-Stokes data of one and the same linear system (\ref{mchoyn}) with parameters 
   $(\chi_0,a_0,s_0)$, defined by the same pair of paths. Thus, two distinct 
   maps $\mct_{\pm}$ take the same value 
   $(\chi_0,a_0,s_0)$   at the point $(s_0,\xi_+)$, and  they differ by signs of $(\chi,a)$-components. Hence, $\chi_0=a_0=0$. The  contradiction thus obtained proves
   Lemma \ref{lreg}.
   \end{proof}
   
   \begin{proposition} \label{prat} The surfaces $S_{\ell,\pm}^o$ are fibered by regular rational curves 
   (finitely punctured Riemann spheres) of degree $2\ell+1$ over the $s$-axis punctured 
   at the origin.  There exists a 
   $d\in\nn$, $d\leq2\ell+1$, such that for all but a discrete subset $Coll_{\ell}\subset\cc^*_s$ of values $s$ the $s$-fiber of the  surface $S_{\ell,\pm}^o$ has exactly  $d$ distinct punctures, and for  $s\in Coll_{\ell}$ the corresponding number of punctures is less than $d$.    
   \end{proposition}
   
   \begin{proof} For every $s\neq0$ let $\mathcal X_{s,\pm}$ denote the $s$-fiber: the preimage under the projection
   $\pi_s:S_{\ell,\pm}^o\to\cc_s$. It is a non-empty algebraic curve. Indeed, it is the zero locus of the polynomial 
    $\mcp_{\ell,\pm}(\chi,a,s)$ with fixed $s$ and variable $(\chi,a)$. The latter polynomial 
    in $(\chi,a)$ is non-constant for every $s$: substituting $\chi=0$ we get the polynomial 
    $Q_{\ell,\pm}(\mu,r)=\det(\mcg_\ell\pm r\idd)$ defining the curve $\Gamma_{\ell,\pm}$ 
    (up to constant factor, see Proposition \ref{det=}), $\mu=\frac s2$, $r=\frac a2$. This is a polynomial of degree $\ell>0$ in $r$, and hence, is non-constant. Thus, the set $\mathcal X_{s,\pm}$ is a non-empty zero locus of a polynomial, 
    and hence, an algebraic curve.   Therefore, it is parametrized (bijectively except for possible singularities) 
   by a punctured compact Riemann surface. The punctures 
 correspond to the intersections of the projective closure $\overline{\mathcal X_{s,\pm}}$ with 
   the   infinity hyperplane, which will be denoted by $\mch$. Each puncture comes with multiplicity equal 
   to the corresponding intersection index of the compact curve 
   $\overline{\mathcal X_{s,\pm}}$ with $\mch$.  Therefore, the total  
   number of punctures with multiplicities is  bounded by the degrees of the curves 
   $\overline{\mathcal X_{s,\pm}}$. Their degrees are no greater than the degree of the polynomials 
   $\mcp_{\ell,\pm}$, which is no greater than $2\ell+1$. 
      
    On the other hand,  the map $\mct_{\pm}$ parametrizes the 
curve $\mathcal X_{s,\pm}$ biholomorphically 
   by an open subset $X_s\subset\oc_\xi$, which is the $s$-fiber  of the open subset 
   $\triv_\ell^o\subset\cc^*\times\oc$. Therefore, $X_s$ should be 
  Riemann sphere punctured in no more than $2\ell+1$ points. 
   Hence, each $\mcx_{s,\pm}$ is a rational 
   curve, and it is regular by biholomorphicity of the parametrization $\mct_{\pm}$. 
      Distinct punctures may collide only at a discrete subset $Coll_\ell\subset\cc^*$ of values $s$. Indeed, 
      resolution of singularities of the projective closure $\overline{S_{\ell,\pm}^o}$ 
      of the surface $S^o_{\ell,\pm}$, which exists by Hironaka Theorem \cite{hironaka}, 
      transforms the (singular) fibration by $s$-fibers into a singular fibration of a regular surface by 
      rational curves. Their punctures correspond to their intersections with another algebraic curve: the image 
      of the intersection of the surface $\overline{S^o_{\ell,\pm}}$ with the infinity plane under the resolution of singularities. 
      Collisions of the above-mentioned intersections obviously correspond to a discrete subset of values $s$. 
      This implies the last statement of Proposition \ref{prat}. Each curve  $\mathcal X_{s,\pm}$ is irreducible, 
      being rational. Let us show that it has  degree $2\ell+1$. Indeed, it is the zero locus of the restriction to 
      the given value $s$ of the polynomial $\mcp_{\ell,\pm}(\chi,a,s)=\det(G_{1,\ell}\pm G_{2,\ell})$. 
      The highest degree monomial of the latter polynomial in the variables $(\chi,a)$ with fixed $s$ 
      is unique and equal to $\chi^{\ell+1}a^\ell$ up to sign, which follows from formulas (\ref{mg1}), (\ref{mg2}) for 
      the matrices $G_{1,\ell}$ and $G_{2,\ell}$. Therefore, the curve 
      $\mcx_{s,\pm}$ has degree $2\ell+1$, unless the polynomial $\mcp_{\ell,\pm}(\chi,a,s)$  with fixed $s\neq0$ has a multiple zero along its irreducible zero locus $\mcx_{s,\pm}$. The latter multiple zero case is impossible. Indeed, in the latter case there would exist another polynomial $\Psi_{s,\pm}(\chi,a)$ vanishing  on $\mcx_{s,\pm}$ and an $m\in\nn$, $m\geq2$, such that $\Psi_{s,\pm}^m(\chi,a)=\mcp_{\ell,\pm}(\chi,a,s)$ up to constant factor. Then the highers homogeneous part of the polynomial $\Psi_{s,\pm}$ is 
      a monomial $c\chi^\alpha a^\beta$, $c\neq0$, $\alpha m=\ell+1$, $\beta m=\ell$, $m(\alpha-\beta)=1$, thus $m=1$, -- a contradiction. The proposition is proved.
   \end{proof} 
   
   Statement 1) of Theorem \ref{tsmooth} follows from Lemma \ref{lreg} and Proposition \ref{disjax}. 
   Its Statement 2) follows 
   from Proposition \ref{prat}. It remains to prove Statement 3). The  
 preimage $\mct_{\pm}^{-1}(S_{\ell,\pm}^o)$ is the complement of the 
   product $\cc^*\times\oc$ to the union $\Sigma_\ell\cup\Sigma_\ell'$. 
   Let us show that the latter union is an analytic hypersurface. We already know that 
   $\Sigma_\ell$ is an analytic hypersurface, and $\Sigma_\ell'$ is an analytic hypersurface in 
   $(\cc^*\times\oc)\setminus\Sigma_\ell$. It remains to show that the closure $\overline{\Sigma_\ell'}$ 
   is an analytic hypersurface (if non-empty), i.e., that its a priori potential singularities contained in $\Sigma_\ell$ are removable. Here by non-removable singularity we mean a point where the analytic subset in question does 
   not extend as an analytic subset: a point where the analytic subset is defined but not regular is considered 
   as removable by definition.  Indeed, it follows from definition and analyticity 
   that the only potential non-removable singularities are points of the intersection 
    $\overline{\Sigma_\ell'}\cap\Sigma_\ell$. The latter intersection points have $s$-coordinates lying in the discrete subset $Coll_\ell$ 
    from Proposition \ref{prat}, and each intersection point  
     corresponds to a puncture: lies in the finite complement  $\oc\setminus X_{s,\pm}$ of the corresponding fiber of $\triv_\ell^o=(\cc^*\times\oc)\setminus(\Sigma_\ell\cup\Sigma_{\ell}')$. Hence, each intersection point is 
     isolated. 
     An isolated intersection point is a removable singularity of the hypersurface $\Sigma_\ell'$, by  Shiffman Theorem, see \cite{shiffman} and \cite[subsection 4.4]{chirka}. Therefore, $\overline{\Sigma_\ell'}$ 
     is a hypersurface, and hence, so is $\Sigma_\ell\cup\Sigma_\ell'$. Theorem \ref{tsmooth} 
     is proved.
     
     \subsection{Irreducibility of polynomials and surfaces. Proof of Theorem \ref{tirred}}
     
     Irreducibility of the surface $S_{\ell,\pm}$ is equivalent to connectivity of its regular part. 
     Its open subset $S_{\ell,\pm}^o=S_{\ell,\pm}\setminus\{ s=0\}$ is connected 
     and lies in its regular part (Lemma \ref{lreg}). Thus, to prove irreducibility, it remains to show that 
     $S_{\ell,\pm}$ does not contain the hyperplane $\{ s=0\}$. 
     
     \begin{proposition} \label{nonvan} The polynomial $\mcp_{\ell,\pm}$ does not vanish identically on the hyperplane $\{ s=0\}$. 
     \end{proposition}
     \begin{proof} The polynomial $\mcp_{\ell,\pm}(0,a,0)$ is a polynomial in one 
     variable $a$ with highest degree monomial $\frac{a^\ell}{2^{\ell+1}}$ up to sign, by formulas 
     (\ref{mg1}), (\ref{mg2}) for the matrices $G_{1,\ell}$ and $G_{2,\ell}$.
              \end{proof}

        The surfaces $S_{\ell,\pm}$ do not contain the hyperplane $\{ s=0\}$, by Proposition \ref{nonvan}. This together with the above argument implies their irreducibility. 
        
        Let us prove irreducibility of the polynomial $\mcp_{\ell,\pm}$. Its zero locus is an irreducible 
        surface $S_{\ell,\pm}$. Thus, it remains to show that $S_{\ell,\pm}$ is not its multiple zero locus. 
        This follows from the fact that for every fixed $s\neq0$ the zero locus of the 
        polynomial $\mcp_{\ell,\pm}(\chi,a,s)$ with fixed $s$ 
        is the $s$-fiber of the surface $S_{\ell,\pm}^o$, and it is not 
        its multiple zero locus, as was shown in the proof of Proposition \ref{prat}.  Theorem \ref{tirred} is proved.      
         
     \subsection{Smoothness of the spectral curves and the genus formula. Proof of Theorems 
     \ref{threg} and \ref{congenu}}
     
     As was shown by I.Netay, Theorem \ref{threg} (smoothness of the curve 
     $\Gamma_{\ell}^o=\Gamma_{\ell}\setminus\{ \mu=0\}\subset\cc^2_{\la,\mu}$) implies 
     the genus formula given by Theorem \ref{congenu}. Thus, for the proof of the genus formula, 
     it remains to prove Theorem \ref{threg}. To do this, we consider the line field on 
     $\cc^3_{\chi,a,s}$ given by the system of the second and third 
differential equations in (\ref{isomnews'}). It is generated by the vector field 
\begin{equation}v:=(v_\chi, v_a, v_s), \ v_\chi=\frac12(a(1-4\chi^2)+2\ell\chi), \ 
v_a=2\chi(a^2-s^2)-\ell a, \ v_s=s,\label{vp3}\end{equation}
which will be called the {\it P3 isomonodromic vector field.} 
 
 \begin{theorem} \label{tisom} Each surface $S_{\ell,\pm}$ is tangent to the P3 
  isomonodromic vector field $v$. 
   \end{theorem}

\begin{proof} The  surface $S_{\ell,\pm}$ is described in terms of monodromy-Stokes data:  
parabolicity and non-degeneracy conditions, see 
Theorem \ref{thpar}, which are invariant under isomonodromic deformations. Therefore, it is tangent to the P3 isomonodromic vector field $v$.\end{proof}

     Smoothness of the curve $\Gamma_\ell^o$ is equivalent to smoothness of one (or equivalently, 
     both) components $\Gamma_{\ell,\pm}^o\subset\cc^2_{a,s}$ of its pullback under the map 
     $(a,s)\mapsto(\la=\frac14(a^2-s^2), \mu=\frac s2)$. The curve $\Gamma_{\ell,\pm}^o$ is 
     the complement of the zero locus $\Gamma_{\ell,\pm}=\{ Q_{\ell,\pm}=0\}$ to the line 
     $\{ s=0\}$. It is the intersection of the smooth surface $S_{\ell,\pm}^o$ with the 
     hyperplane $\{\chi=0\}$. The latter intersection is transversal, since the line field 
     (\ref{isomnews'}) is tangent to $S_{\ell,\pm}^o$ and transversal to the hyperplane 
     $\{\chi=0\}$ outside the line $\{ a=\chi=0\}$. But we already know that 
     $S_{\ell,\pm}^o$ is disjoint from the line $\{\chi=a=0\}$, by Theorem \ref{tsmooth}, Statement 1). Therefore, the line field on $\cc^3_{\chi,a,s}$ given by (\ref{isomnews'}) is transversal to the hyperplane $\{\chi=0\}$ 
     at points of its intersection with $S_{\ell,\pm}^o$, which are exactly the points of the curve 
     $\Gamma_{\ell,\pm}^o$, by Proposition \ref{det=}. Thus, their intersection is transversal. 
     Therefore, the curve $\Gamma_{\ell,\pm}^o$ is smooth, and hence, so is $\Gamma_{\ell}^o$. 
     Theorem \ref{threg}, and hence, Theorem \ref{congenu} are proved.

\section{Isomonodromic foliations of the determinantal surfaces}
Here we study  the isomonodromic foliations of the complex surfaces $S_{\ell,\pm}$ for $\ell=0,1$: foliations by phase curves of the P3 isomonodromic vector field $v$ given by (\ref{vp3}). 

Theorem \ref{tisom} implies the following corollary.

\begin{corollary} For every $\ell\in\nn$ one has 
  \begin{equation}\frac{d\mcp_{\ell,\pm}}{dv}=h_{\ell,\pm}\mcp_{\ell,\pm}, \ h_{\ell, \pm}=h_{\ell,\pm}(\chi,a,s) 
  \ \text{ is a polynomial.}\label{derpol}\end{equation}
\end{corollary}
\begin{proof} The surface $S_{\ell,\pm}$ being zero locus of 
an irreducible polynomial $\mcp_{\ell,\pm}$, and the vector field $v$ being  polynomial and tangent to 
$S_{\ell,\pm}$ (Theorem \ref{tisom}), we get (\ref{derpol}). 
\end{proof}

\begin{problem} For every $\ell$ find the polynomial multiplier $h_{\ell,\pm}$ from (\ref{derpol}).
\end{problem}

Below we present formulas for the polynomial multipliers $h_{\ell,\pm}$ for $\ell=1,2$. 

\subsection{Case  $\ell=0$}

 For $\ell=0$ the set of those parameters $(\chi,a,s)$ for which the 
corresponding system (\ref{mchoyn}) has a vector polynomial solution is  the union of two planes 
$\{\chi=\pm\frac12\}$, both being equipped with the coordinates $(a,s)$. The isomonodromic 
foliations of the latter planes are given by the Riccati equation derived from the second equation in 
(\ref{isomnews'}): 
$$a'_s=\pm\frac1s(a^2-s^2).$$

\subsection{Case $\ell=1$}
Let us calculate $h_{1,\pm}$. Afterwards we show that the corresponding isomonodromic foliation is given by a Riccati  equation. We also present a result on monodromy and Stokes operators of the 
linear system whose projectivization is the latter Riccati equation.
\begin{proposition} For $\ell=1$ one has 
\begin{equation}\frac{d\mcp_{1,\pm}}{dv}=h_{1,\pm}\mcp_{1,\pm}, \ \ h_{1,\pm}(\chi,a,s)=(1-2\chi(a\pm s)).\label{h1}\end{equation}
\end{proposition}
\begin{proof} Let us multiply $\mcp_{1,\pm}$ by $\pm4$. We get  
$$\wt\mcp_{1,\pm}:=\pm4\mcp_{1,\pm}=(a\pm s)(1-4\chi^2)+4\chi,$$
  $$v_a\frac{\partial\wt\mcp_{1,\pm}}{\partial a}+v_s\frac{\partial\wt\mcp_{1,\pm}}{\partial s}=
  (1-4\chi^2)(2\chi(a^2-s^2)- a\pm s)$$
  \begin{equation}=(\wt\mcp_{1,\pm}-4\chi)(2\chi(a\mp s)+1)-2a(1-4\chi^2),\label{mcp11}\end{equation}
  $$v_\chi\frac{\partial\wt\mcp_{1,\pm}}{\partial \chi}=(-8\chi(a\pm s)+4)\frac12(a(1-4\chi^2)+2\chi)=(-4\chi(a\pm s)+2)(a(1-4\chi^2)+2\chi)$$
  $$=-8\chi^2(a\pm s)+4\chi+2a(1-4\chi^2)-4a\chi(1-4\chi^2)(a\pm s)$$
  \begin{equation}=-8\chi^2(a\pm s)+4\chi+2a(1-4\chi^2)-4a\chi(\wt\mcp_{1,\pm}-4\chi).
  \label{mcp12}\end{equation}
  Summing expressions (\ref{mcp12}) and (\ref{mcp11}) yields
  $$\frac{d\wt\mcp_{1,\pm}}{dv}=-8\chi^2a\mp 8\chi^2s+4\chi+\wt\mcp_{1,\pm}(1-2\chi(a\pm s))+16a\chi^2-4\chi$$
  $$-8\chi^2a\pm8\chi^2s=\wt\mcp_{1,\pm}(1-2\chi(a\pm s)).$$
  This proves (\ref{h1}).
  \end{proof}
  
  In what follows we treat only the case "$-$";  the case "$+$" is then obtained by the sign change 
  $(\chi,a)\mapsto(-\chi,-a)$. 
\begin{proposition} 
In the coordinates $(\chi,s)$ on the surface $\gga_{1,-}=\{\mcp_{1,-}=0\}$ 
the vector field (\ref{vp3}) generates the line field defined by the Riccati equation
\begin{equation}\chi'_s=\frac{1-4\chi^2}2-\frac{\chi}s.\label{ricchis}\end{equation}
\end{proposition}
\begin{proof} The equation $\mcp_{1,-}=0$ is equivalent to the formula
\begin{equation} a=s-\frac{4\chi}{1-4\chi^2}.\label{forma}\end{equation}
Substituting (\ref{forma}) 
to the first differential equation in (\ref{isomnews'})  yields
$$\chi'_s=\frac{a-2\chi(-1+2\chi a)}{2s}=\frac{a(1-4\chi^2)+2\chi}{2s}=\frac{s(1-4\chi^2)-2\chi}{2s}=\frac12(1-4\chi^2)-\frac{\chi}s.$$
This proves the statement of the proposition.
\end{proof}

 The Riccati equation (\ref{ricchis}) is the projectivization of the linear system
\begin{equation} \Psi'_s=\left(\frac1s\left(\begin{matrix} 0 & 0\\ 0 & -1\end{matrix}\right)+\left(\begin{matrix} 0 & 2\\ \frac12 & 0\end{matrix}\right)\right)\Psi: 
\label{psisys}\end{equation}
a function $\chi(s)$ is a solution of (\ref{ricchis}), if and only if it can be presented as a ratio $\frac{\Psi_2}{\Psi_1}(s)$, where $\Psi(s)=(\Psi_1(s),\Psi_2(s))$ is a 
vector solution of (\ref{psisys}). 

System (\ref{psisys}) has Fuchsian singularity at the origin with residue matrix 
$\diag(0,-1)$ and irregular singularity at infinity of Poincar\'e rank 1.
\begin{proposition} The formal normal form at infinity of system (\ref{psisys}) is 
\begin{equation}
\wt\Psi'_s=\left(\frac1s\left(\begin{matrix} -\frac12 & 0\\ 0 & -\frac12\end{matrix}\right)+\left(\begin{matrix} 1 & 0\\ 0 & -1\end{matrix}\right)\right)\wt\Psi. 
\label{psisform}\end{equation}

2) The Stokes matrices of system (\ref{psisys}) at infinity are both non-trivial.

3) Its monodromy along circuit around the origin is  unipotent Jordan cell. 
\end{proposition}

\begin{proof}  The free term matrix in (\ref{psisys})  has eigenvalues $1$, $-1$;  the corresponding eigenvectors are 
$(2,1)$, $(2,-1)$. Therefore, it  is diagonalized by conjugacy as follows:
$$\left(\begin{matrix}2 & 2\\ 1 & -1\end{matrix}\right)^{-1}\left(\begin{matrix} 0 & 2\\ \frac12 & 0\end{matrix}\right)\left(\begin{matrix}2 & 2\\ 1 & -1\end{matrix}\right)=\left(\begin{matrix} 1 & 0\\ 0 & -1\end{matrix}\right).$$
The same conjugacy brings the residue matrix in (\ref{psisys}) to 
$$\left(\begin{matrix}2 & 2\\ 1 & -1\end{matrix}\right)^{-1}\left(\begin{matrix} 0 & 0\\ 0 & -1\end{matrix}\right)\left(\begin{matrix}2 & 2\\ 1 & -1\end{matrix}\right)
=\left(\begin{matrix} -\frac12 & \frac12\\ \frac12 & -\frac12\end{matrix}\right).$$
This implies that the residue matrix of the formal normal form at infinity is $\diag(-\frac12, -\frac12)$. Statement 1) is proved.

The eigenvalues of the monodromy operator of system (\ref{psisys}) along a positive circuit around the origin are equal to 
the exponents of $2\pi i$ times the residue matrix eigenvalues $0$, $-1$, since the singular point at the origin is Fuchsian with 
residue matrix $\diag(0,-1)$. Thus, both monodromy eigenvalues are equal to one, the monodromy is unipotent, and its trace is equal to 2. On the other hand, the 
monodromy matrix in an appropriate canonical sectorial solution basis at infinity is the product of three matrices: the   monodromy matrix of the formal normal form 
$\exp(-2\pi i\diag(\frac12,\frac12))=-Id$ 
and the inverses of the Stokes matrices; the latter inverses  are unipotent respectively lower and upper triangular of the type
$$C_1^{-1}=\left(\begin{matrix} 1 & 0\\ -c_1 & 1\end{matrix}\right), \ C_0^{-1}=\left(\begin{matrix} 1 & -c_0\\ 0 & 1\end{matrix}\right); \ 
c_{0,1} \text{ are called the {\it Stokes multipliers,}}$$
see (\ref{monprod}). 
The total product of the three matrices in question is equal to 
\begin{equation}M=\left(\begin{matrix} -1 & c_0\\ c_1 & -(1+c_0c_1)\end{matrix}\right).\label{monst}\end{equation}
Its trace is equal to $-(2+c_0c_1)$. But on the other hand, it is equal to 2, see the above discussion. Thus, 
$c_0c_1=-4$. This proves Statement 2) of the proposition. Statement 3), which is equivalent to nontriviality of the monodromy (unipotence), 
follows from non-triviality of the Stokes multipliers and (\ref{monst}), as in \cite{4}. The proposition is proved.
\end{proof} 

\subsection{Case $\ell=2$}

\begin{proposition} For $\ell=2$ one has 
\begin{equation}\frac{d\mcp_{2,\pm}}{dv}=h_{2,\pm}\mcp_{2,\pm}, \ \ h_{2,\pm}(\chi,a,s)=(2-a(2\chi\mp1)).\label{h2}\end{equation}
\end{proposition}
\begin{proof} Set $\wt\mcp_{2,\pm}:=\pm 8\mcp_{2,\pm}$. One has 
$$\wt\mcp_{2,\pm}=(2\chi\pm1)^2(2\chi\mp1)(a^2-s^2)-2a(2\chi\pm1)(6\chi\mp1)+16\chi,$$
see (\ref{mcp2}),  
$$v_s\frac{\partial\wt\mcp_{2,\pm}}{\partial s}=-2s^2(2\chi\pm1)^2(2\chi\mp1)$$
\begin{equation}=2(a^2-s^2)(2\chi\pm1)^2(2\chi\mp1)-2a^2(2\chi\pm1)^2(2\chi\mp1),
\label{mcp21}\end{equation}
$$v_a\frac{\partial\wt\mcp_{2,\pm}}{\partial a}=(2\chi(a^2-s^2)-2a)(2\chi\pm1)(2a(4\chi^2-1)-12\chi
\pm2),$$
$$v_{\chi}\frac{\partial\wt\mcp_{2,\pm}}{\partial\chi}=\frac12(a(1-4\chi^2)+4\chi)((a^2-s^2)(2\chi\pm1)(12\chi\mp2)-8a(6\chi\pm1)+16).$$
Therefore 
\begin{equation}\frac{d\wt\mcp_{2,\pm}}{dv}=v_s\frac{\partial\wt\mcp_{2,\pm}}{\partial s}+ 
v_a\frac{\partial\wt\mcp_{2,\pm}}{\partial a}+v_\chi\frac{\partial\wt\mcp_{2,\pm}}{\partial\chi}=
A(a^2-s^2)+B,\label{forAB}\end{equation}
$$A=(2\chi\pm1)(2(4\chi^2-1)+4\chi a(4\chi^2-1) -24\chi^2\pm4\chi+ (a(1-4\chi^2)+4\chi)(6\chi\mp1))
$$
\begin{equation}=(2\chi\pm1)(4\chi^2-1)(2-a(2\chi\mp1)).\label{formA}\end{equation}
$$B=-2a^2(2\chi\pm1)^2(2\chi\mp1)-4a(2\chi\pm1)(a(4\chi^2-1)-6\chi\pm1)+
4(a(1-4\chi^2)+4\chi)(-a(6\chi\pm1)+2)$$
\begin{equation}=2a^2(4\chi^2-1)(6\chi\mp1)+4a(-20\chi^2+1)+32\chi.\label{formB}\end{equation}
In order to prove (\ref{h2}), taking into account formula (\ref{formA}), it remains to show that 
\begin{equation}B=(2-a(2\chi\mp1))(-2a(2\chi\pm1)(6\chi\mp1)+16\chi).\label{bforn}\end{equation}
Opening brackets in (\ref{bforn}), one gets (\ref{formB}). This proves (\ref{bforn}), and hence,  
 (\ref{h2}).
\end{proof}

\section{Acknowledgements}

I am grateful to Yulia Bibilo, Victor Buchstaber, Igor Netay, Victor Novokshenov, Constantin Shramov, 
Vladimir Sokolov   and  Ilya Vyugin for helpful discussions.

\end{document}